\documentclass[a4paper,10.5pt,reqno]{amsart}

\usepackage[utf8]{inputenc}
\usepackage{amsmath}
\usepackage{amssymb}
\usepackage{graphicx} 
\usepackage{enumerate}
\usepackage[noadjust]{cite}
\allowdisplaybreaks
\usepackage[toc,page]{appendix}
\usepackage{hyperref}    
\usepackage{tikz-cd}
\usepackage{bbm}
\usepackage{authblk}

    {\cleardoublepage\thispagestyle{empty}\null\vfill\begin{center}%
    \bfseries Acknowledgements\end{center}}%
    {\vfill\null}

\usepackage{enumitem}
\setlistdepth{20}

\newcommand{\N}{\mathcal{N}}
\newcommand{\R}{\mathcal{R}}
\renewcommand{\O}{\mathcal{O}}
\newcommand{\D}{\mathcal{D}}
\renewcommand{\S}{\mathcal{S}}
\newcommand{\Q}{\mathcal{Q}}
\newcommand{\Sym}{\textup{Sym}}
\newcommand{\ZZ}{\mathbb{Z}}

\newcommand{\RR}{\mathbb{R}}
\newcommand{\QQ}{\mathbb{Q}}
\newcommand{\CC}{\mathbb{C}}

\newcommand{\FF}{\mathbb{F}}
\newcommand{\PP}{\mathbb{P}}
\newcommand{\GG}{\mathbb{G}}
\newcommand{\DDpt}{\mathbb D_{\textup{PT}}^{X}}
\newcommand{\DD}{\mathbb D}
\renewcommand{\L}{\mathcal L}
\newcommand{\1}{\mathbbm{1}}
\newcommand{\p}{\mathsf{p}}

\newcommand{\ch}{\textup{ch}}
\newcommand{\td}{\textup{td}}

\newcounter{conta}
\newtheorem{lemma}[conta]{Lemma}
\newtheorem{theorem}[conta]{Theorem}
\newtheorem{proposition}[conta]{Proposition}
\newtheorem{corollary}[conta]{Corollary}
\newtheorem{claim}[conta]{Claim}
\newtheorem{conjecture}{Conjecture}

\newtheorem{definition}{Definition}
\theoremstyle{remark}
\newtheorem{remark}{Remark}
\theoremstyle{remark}

\begin{document}
\title[Virasoro constraints for stable pairs]{\normalfont{Virasoro conjecture for the stable pairs descendent theory
of simply connected 3-folds}\vspace{0.1cm}\\ \footnotesize (with applications to the Hilbert
scheme of points of a surface)\\ \vspace{0.5cm}
\small{Miguel Moreira}}
\address {ETH Z\"urich, Department of Mathematics}
\email{miguel.moreira@math.ethz.ch}

\maketitle

\begin{abstract}
This paper concerns the recent Virasoro conjecture for the theory of stable pairs on a 3-fold proposed by Oblomkov, Okounkov, Pandharipande and the author in \cite{virasorotoricPT}. Here we extend the conjecture to 3-folds with non-$(p,p)$-cohomology and we prove it in two specializations.

For the first specialization, we let $S$ be a simply-connected surface and consider the moduli space $P_n(S\times \PP^1, n[\PP^1])$, which happens to be isomorphic to the Hilbert scheme $S^{[n]}$ of $n$ points on $S$. The Virasoro constraints for stable pairs, in this case, can be formulated entirely in terms of descendents in the Hilbert scheme of points. The two main ingredients of the proof are the toric case and the existence of universal formulas for integrals of descendents on $S^{[n]}$. The second specialization consists in taking the 3-fold $X$ to be a cubic and the curve class $\beta$ to be the line class. In this case we compute the full theory of stable pairs using the geometry of the Fano variety of lines. 
\end{abstract}

\section{Introduction}

\subsection{Stable pairs}

\noindent Let $X$ be a smooth projective 3-fold over $\CC$. A stable pair on $X$ is a coherent sheaf $F$ on $X$ together with a section $s: \O_X\to F$ satisfying the following two stability conditions:
\begin{enumerate}
\item $F$ is pure of dimension 1, i.e. every non-trivial coherent sub-sheaf of $F$ has dimension\footnote{Dimension of a coherent sheaf means the dimension of its support.} 1.
\item The cokernel of $s$ has dimension $0$.
\end{enumerate}

We can associate two discrete invariants to a stable pair, namely 
\[n=\chi(X, F)\in \ZZ \textup{ and }\beta=[C]\in H_2(X; \ZZ)\]
where $C$ is the support of $F$. There is a projective fine moduli space $P_n(X, \beta)$ parametrizing stable pairs with fixed discrete invariants $n$ and $\beta$. Moreover this space carries an obstruction theory and a virtual fundamental class\footnote{Unless otherwise specifies, homology and cohomology are understood with rational coefficients, which are enough for our purposes. However, the virtual fundamental class is actually constructed in the integral Chow ring $A_{d_\beta}(P_n(X, \beta); \ZZ)$.}
\[[P_n(X, \beta)]^{\textup{vir}}\in H_{2 d_\beta}(P_n(X, \beta))\]
where 
\[d_\beta=\int_\beta c_1(X)\]
is the (complex) virtual dimension of $P_n(X, \beta)$ -- note that, unlike in Gromov-Witten theory, the virtual dimension doesn't depend on $n$. See \cite{PTfoundations} for the construction of the virtual fundamental class.

Over $X\times P_n(X, \beta)$ we have the universal stable pair $\O_{X\times P_n(X, \beta)}\to \FF$; when restricted to a fiber $X\times (F, s)$, the universal stable pair is canonically isomorphic to $s: \O_X\to F$. We use this universal structure to define tautological descendent classes. Denote by $\pi_X: X\times P_n(X, \beta)\to X$ and $\pi_P: X\times P_n(X, \beta)\to P_n(X, \beta)$ the projections onto the first and second factors, respectively. 

\begin{definition}
Given $\gamma\in H^\ast(X)$ and $k\in \ZZ_{\geq 0}$, we define
\begin{equation}\ch_k(\gamma)=(\pi_P)_\ast\left(\ch_k\left(\FF-\mathcal O_{X\times P_n(X, \beta)}\right)\cdot\pi_X^\ast(\gamma)\right)\in H^{\ast}(P_n(X, \beta)).
\end{equation}
\end{definition}
Because $\FF$ is supported in codimension $2$, the Chern character $\ch_k(\FF)$ vanishes for $k=0, 1$, hence
\begin{equation}
\label{collapsech}\ch_0(\gamma)=-\int_{X}\gamma \in \QQ\cong H^0(P_n(X, \beta)) \textup{ and } \ch_1(\gamma)=0;
\end{equation}
in particular, $\ch_0(\gamma)$ vanishes if $\gamma\in H^{<6}(X)$. 

Note that if $\gamma$ has (real cohomological) degree $d$ then $\ch_k(\gamma)$ has (real cohomological) degree $d+2k-6$; moreover, if $\gamma$ has Hodge degree $(p,q)$ then the descendent $\ch_k(\gamma)$ has Hodge degree $(p+k-3, q+k-3)$. Alternatively, the descendents may be defined by their action on $H_\ast(P_n(X, \beta))$, which by the push-pull formula is
\[(\pi_P)_\ast\left(\ch_k(\FF-\O_{X\times P_n(X, \beta)})\cdot \pi_X^\ast(\gamma)\cap \pi_P^\ast(\,\cdot\,)\right).\]

Given a product of descendent classes $D=\prod_{j=1}^m \ch_{k_j}(\gamma_j)$, we denote integration against the virtual fundamental class by
\begin{equation}
\left\langle D \right \rangle_{n, \beta}^{X, \textup{PT}}=\int_{[P_n(X, \beta)]^{\textup{vir}}}D.
\end{equation}
The generating function of these invariants is called the partition function and denoted by
\begin{equation}
Z_{\textup{PT}}^X(q \mid D)_\beta=\sum_{n\in \ZZ}q^n\left\langle D \right \rangle_{n, \beta}^{X, \textup{PT}}\in \QQ((q));
\end{equation}
the series lies in $\QQ((q))$ because $P_n(X, \beta)$ is empty for very small $n$. We omit $X, \beta$ from the notation if it's clear from the context. 

\begin{conjecture}
\label{conj: rationality}
For any product of descendents $D=\prod_{k=1}^m \ch_{k_j}(\gamma_j)$, the Laurent series $Z_{\textup{PT}}(q\mid D)$ is the Laurent expansion of a rational function in $q$ satisfying the following functional equation:
\[Z_{\textup{PT}}(q^{-1}\mid D)=(-1)^{\sum_{j=1}^m k_j}q^{-d_\beta}Z_{\textup{PT}}(q\mid D).\]
\end{conjecture}

We refer to \cite{survey} for a survey of partial results in the direction of these conjectures, as well as a discussion of equivariant and relative versions.

It's widely believed that the theory of stable pairs is equivalent to other curve counting theories on $3$-folds, such as Gromow-Witten and Donaldson-Thomas theories, \cite{mnop1, mnop2}. Precise statements are known for toric varieties by work of Pandharipande and Pixton, \cite{toricGWPT}; simpler formulas for the correspondence were found more recently by Oblomkov, Okounkov and Pandharipande in \cite{toricGWPTvertex, virasorotoricPT}.

\subsection{String, divisor and dilaton equations for stable pairs}

\noindent The string, divisor and dilaton equations of Gromov-Witten theory have parallel incarnations in the stable pairs side. Their stable pairs versions take a much simpler form and can be formulated as expressions for descendents in non-positive degree.

\begin{proposition}
\label{descendentequations}
For any smooth projective 3-fold $X$, $\beta\in H_2(X; \ZZ)$ and $n\in \ZZ$, we have the following identities of descendents:
\begin{enumerate} 
\item[i)] $\ch_2(\gamma)=0$ for any $\gamma\in H^{p,0}(X)$ or $\gamma\in H^{0,q}(X)$
\item[ii)] $\ch_2(\delta)=\int_\beta \delta$ for any $\delta \in H^2(X)$
\item[iii)] $\ch_3(1)=n-\frac{d_\beta}{2}$.
\end{enumerate}
\end{proposition}
The second and third equations are understood in $H^0(P_n(X, \beta))\cong \QQ$. Equations $i)$ (in the case $\gamma=1$), $ii)$ and $iii)$ are known as the string equation, divisor equation and dilaton equation, respectively. They are well known (see for instance \cite[Section 3.2]{survey}), but we include a proof for completeness.
\begin{proof}
Let $Z\subseteq X\times P_n(X, \beta)$ be the support of $\FF$. Set theoretically
\[Z=\{(x, (F, s))\in X\times P_n(X, \beta): F_x\neq 0\}.\]
Since $\O_{Z}\to \FF_{|Z}$ has cokernel supported in codimension 1, $\FF_{|Z}$ is a coherent sheaf of rank 1 in $Z$. Thus by Grothendieck-Riemann-Roch we get $\ch_2(\FF)=[Z]\in H^4(X\times P_n(X, \beta))$. Hence 
\[\ch_2(\gamma)=(\pi_P)_\ast([Z]\pi_X^\ast \gamma)=(\pi_P^Z)_\ast(\pi_X^Z)^\ast \gamma\]
where $\pi_P^Z$ and $\pi_X^Z$ are the projections of $Z$ onto $P_n(X, \beta)$ and $X$, respectively. The string equation follows immediately because $(\pi_P^Z)_\ast$ reduces the Hodge grading by $(1,1)$. For the divisor equation, integration along the fibers gives
\[\ch_2(\delta)=\int_C \delta=\int_\beta \delta\]
where $C=(\pi_{P}^Z)^{-1}(F,s)$ is the support of $F$.

Finally, for the dilaton equation we have
\[\ch_3(1)=(\pi_P)_\ast \ch_3(\FF)=\int_X j^\ast \ch_3(\FF)=\int_X \ch_3(F)\]
where $j: X\cong X\times (F, s)\hookrightarrow X\times P_n(X, \beta)$ is the inclusion of some fiber. By Hirzebruch-Riemann-Roch and using the facts that $c_j(F)=0$ for $j=0,1$ and $\ch_2(F)=[C]=\beta$ where $C$ is the support of $F$ we get
\[n=\chi(F)=\int_X \ch_3(F)+\frac{1}{2}\int_\beta c_1(X)\]
and thus
\[\ch_3(F)=n-\frac{d_\beta}{2}.\] 
\end{proof}

We may formulate the equations of proposition \ref{descendentequations} in terms of the partition function:
\begin{enumerate}
\item[\textit{i)}] $Z_{\textup{PT}}^X\left(q\mid \ch_2(1) D\right)_\beta=0,$
\item[\textit{ii)}] $Z_{\textup{PT}}^X\left(q\mid \ch_2(\delta) D\right)_\beta=\left(\int_\beta \delta\right)Z_{\textup{PT}}^X\left(q\mid  D\right)_\beta,$
\item[\textit{iii)}] $Z_{\textup{PT}}^X\left(q\mid \ch_3(1) D\right)_\beta=\left(q\frac{d}{dq}-\frac{d_\beta}{2}\right)Z_{\textup{PT}}^X\left(q\mid D\right)_\beta.$
\end{enumerate}
These hold for any $\delta\in H^2(X)$ and for an arbitrary product of descendents $D$.

\subsection{Virasoro constraints for stable pairs}

\noindent The existence of universal equations satisfied by the theory of stable pairs of any 3-fold $X$, parallel to the Virasoro constraints for the Gromov-Witten theory of $X$, was first conjectured in \cite{survey}. By using explicit calculations in $\PP^3$, Oblomkov, Okounkov and Pandharipande guessed the explicit equations for $\PP^3$. More recently a general conjecture was proposed for 3-folds with only $(p,p)$-cohomology and proven for toric 3-folds in the stationary\footnote{Stationary descendents are descendents $\ch_k(\gamma)$ of classes $\gamma\in H^{\geq 2}(X)$.} case. 

We briefly describe the proposed conjecture here in slightly more generality by allowing 3-folds with non-$(p,p)$-cohomology, and in particular we allow odd cohomology. To state the Virasoro conjecture for stable pairs we introduce the formal supercommutative $\QQ$-algebra $\DDpt$ generated by
\[\{\ch_k(\gamma): k\geq 0, \gamma\in H^\ast(X)\}\] 
with the linearity relations
\[\ch_k(\lambda_1\gamma_1+\lambda_2 \gamma_2)=\lambda_1\ch_k(\gamma_1)+\lambda_2\ch_k(\gamma_2).\]

We will write $\ch_k(\gamma)$ both for the generator in the abstract algebra $\DDpt$ and for its geometric realization in $H^\ast(P_n(X, \beta))$ defined earlier. Note that for instance $\ch_1(\gamma)$ is non zero in $\DDpt$ but its geometric realization is zero. However, we'll frequently simplify expressions by replacing $\ch_0(\gamma), \ch_1(\gamma)$ using equations \eqref{collapsech}; when we do so we say we ``collapsed'' $\ch_0, \ch_1$. 

We have the cohomological grading in $\DDpt$: a generator $\ch_k(\gamma)$ has degree $|\gamma|+2k-6$. For each $n\geq 0$, integration against the virtual fundamental class of the geometric realization of an element of $\DDpt$ gives a linear map
\[\langle\cdot \rangle^{X,\textup{PT}}_{n,\beta}: \DDpt\to \QQ.\]

We will define some operators on the algebra $\DDpt$.

\begin{itemize}
\item For $k\geq -1$, define a derivation $R_k$ on $\DDpt$ by fixing its action on the generators: given $\gamma\in H^{p,q}(X)$, let
\begin{equation}R_k(\ch_i(\gamma))=\left(\prod_{j=0}^k(i+p-3+j)\right)\ch_{i+k}(\gamma).\end{equation}
\item The operator $T_k: \DDpt\to \DDpt$ is multiplication by a fixed element of $\DDpt$:
\begin{align}\nonumber T_k&=-\frac{1}{2}\sum_{a+b=k+2}(-1)^{p^Lp^R}(a+p^L-3)!(b+p^R-3)!\ch_a\ch_b(c_1)\\\
&+\frac{1}{24}\sum_{a+b=k}a!b!\ch_a\ch_b(c_1c_2).\end{align}
We are using the abbreviation 
\[(-1)^{p^Lp^R}(a+p^L-3)!(b+p^R-3)!\ch_a\ch_b(c_1)\]
for
\[\sum_{i}(-1)^{p^L_i p^R_i}(a+p^L_i-3)!(b+p^R_i-3)!\ch_a(\gamma^L_i)\ch_b(\gamma^R_i)\]
where $\sum_{i}\gamma^L_i\otimes \gamma^R_i$ is the Kunneth decomposition of $\Delta_\ast c_1\in H^\ast(X\times X)$ ($\Delta: X\to X\times X$ is the diagonal map) and $\gamma^L_i$ and $\gamma^R_i$ have Hodge type $(p^L_i, q^L_i)$ and $(p^R_i, q^R_i)$, respectively.
\item For $\alpha \in H^\ast(X)$, we define the derivation $R_{-1}[\alpha]: \DDpt\to \DDpt$ by its action on the generators:
\[R_{-1}[\alpha](\ch_i(\gamma))=\ch_{i-1}(\alpha\gamma).\]
In particular, $R_{-1}[1]=R_{-1}$. For $k\geq -1$, we define the operators $S_k: \DDpt\to \DDpt$ by
\[S_k=(k+1)!\sum_{p_i^L=0} R_{-1}[\gamma_i^L]\ch_{k+1}(\gamma^R_i).\]
The sum runs over the terms $\gamma_i^L\otimes \gamma_i^R$ of the Kunneth decomposition of $\Delta\in H^\ast(X\times X)$ such that $p_i^L=0$. The operators being summed up are the composition of two operators: first we multiply by $\ch_{k+1}(\gamma^R_i)$ and then we apply the derivation $R_{-1}[\gamma_i^L]$ to the product.
\end{itemize}

Note that if $h^{0,1}=h^{0,2}=h^{0,3}=0$ then $S_k$ is simply $(k+1)!R_{-1}\ch_{k+1}(\p)$. For $k=-1$, after collapsing $\ch_0, \ch_1$, we have $S_{-1}=\ch_0(\p)R_{-1}=-R_{-1}$. For $k=0$ we have 
\[S_0=\sum_{p^L_i=0}\ch_{0}\left(\gamma_i^L\gamma_i^R\right)=\frac{1}{24}\ch_0(c_1c_2)=-\frac{1}{24}\int_X c_1c_2.\]
The above formula is a variation of lemma \ref{lem: identitieshrr} for 3-folds; note that $\td_3(X)=\frac{1}{24}c_1c_2$.

\begin{definition}
\label{def: virasoroop}Let $X$ be a smooth projective 3-fold. For $k\geq -1$, we define the operator $\L_k: \DDpt\to \DDpt$ by
\begin{equation}\L_k=R_k+T_k+S_k.
\end{equation}
\end{definition}

\begin{conjecture}
\label{conj: virasoro}
Let $X$ be a (simply-connected) projective smooth 3-fold. For all $k\geq -1$, $\beta\in H_2(X; \ZZ)$, $n\in \ZZ$ and $D\in \DDpt$ we have
\[\left\langle \L_k(D)\right\rangle_{n, \beta}^{X, \textup{PT}}=0.\]
\end{conjecture}

Similarly to the Virasoro conjecture on Gromov-Witten theory, the cases $k=-1, 0$ of the Virasoro conjecture are easy: $k=-1$ is a formal consequences of the rules for collapsing $\ch_0, \ch_1$; the case $k=0$ uses the divisor equation (proposition \ref{descendentequations}) and the vanishing of invariants whenever the Hodge degrees of the integrand don't match $(d_\beta, d_\beta)$ (note that the virtual fundamental class is algebraic).

\begin{proposition}
The Virasoro conjecture for stable pairs holds for $k=-1, 0$: for any projective smooth 3-fold and any $D\in \DDpt$, we have 
\[\left\langle \L_{-1}(D)\right\rangle_{n, \beta}^{X, \textup{PT}}=0 \textup{ and } \left\langle \L_0(D)\right\rangle_{n, \beta}^{X, \textup{PT}}.\]
\end{proposition}

Strong evidence for the Virasoro conjecture is provided in \cite{virasorotoricPT}, where it's shown that the Virasoro operators on the stable pairs side and on the Gromov-Witten side are intertwined by the conjectural (stationary) GW/PT correspondence. 

\begin{theorem}[Theorem 5 in \cite{virasorotoricPT}]
\label{virGWimpliesvirPT}
Let \(X\) be a projective smooth 3-fold with only $(p,p)$-cohomology for which 
the following two properties are satisfied:
\begin{enumerate}
\item[(i)]
The stationary Virasoro constraints for the Gromov-Witten theory of $X$ hold.
\item[(ii)]
The stationary $\textup{GW}/\textup{PT}$ correspondence holds (see \cite[Section 0.6]{virasorotoricPT}). 
\end{enumerate}
Then, the stationary
Virasoro
constraints for the stable pairs theory of $X$ in conjecture \ref{conj: virasoro} hold.
\end{theorem}

Both the GW/PT correspondence and the Virasoro conjecture for Gromov-Witten are known for toric 3-folds, the first by work of Oblomkov, Okounkov, Pixton and Pandharipande \cite{toricGWPT, toricGWPTvertex} and the latter by Givental \cite{givental}. 

\begin{theorem}[Theorem 4 in \cite{virasorotoricPT}]
Conjecture \ref{conj: virasoro} holds when $X$ is a toric 3-fold and $D$ is stationary.
\end{theorem}

\begin{remark}
We may define a modification of the Virasoro operators $\overline \L_k$ via the Hodge symmetry $(p,q)\leftrightarrow (q,p)$, i.e. by replacing all the appearances of the first Hodge index in the definition of $\L_k$ by the second Hodge index. See \cite[Section 2.10.]{getzler} for a similar observation in Gromov-Witten theory. A simple formula for the commutator $[\L_k, \overline \L_\ell]$ doesn't seem to exist.
\end{remark}

\subsection{Vanishing of descendents of $(p,0)$-classes } \noindent Let $\gamma\in H^{p,0}(X)$ with $p=2$ or $p=3$. Proposition \ref{descendentequations} shows that $\ch_2(\gamma)=0$ in $H^\ast(P_n(X,\beta))$. However, conjecture \ref{conj: virasoro} implies a much more general and surprising vanishing. Let $D\in \DDpt$ be arbitrary. We have
\[[\L_k, \ch_2(\gamma)]=R_k(\ch_2(\gamma))=\frac{(p-1+k)!}{(p-2)!}\ch_{2+k}(\gamma).\]
Since $\ch_2(\gamma)=0$ in $H^\ast(P_n(X, \beta))$, and assuming conjecture \ref{conj: virasoro} is true, it follows that
\[\left\langle \ch_{2+k}(\gamma)D\right\rangle_{n, \beta}^{X, \textup{PT}}=0\]
for every $k\geq -1$, $\gamma$ of type $(p,0)$ with $p\geq 2$ and $D\in \DDpt$.

\begin{conjecture}
\label{conj: vanishing}
Let $X$ be a simply-connected projective smooth 3-fold, let $k\geq 0$ and let $\gamma\in H^{p,0}(X)$ with $p=2$ or $p=3$. Then, for $D\in \DDpt$,
\[\left\langle \ch_{k}(\gamma)D\right\rangle_{n, \beta}^{X, \textup{PT}}=0.\]
\end{conjecture}

One can also speculate that this numerical vanishing holds due to a stronger vanishing at the level of cycles:
\[\ch_{k}(\gamma)\cap [P_n(X, \beta)]^{\textup{vir}}=0.\]

\subsection{Virasoro for the Hilbert scheme of points on a surface}

\noindent One interesting specialization of the Virasoro constraints for the moduli space of stable pairs produces Virasoro constraints on the Hilbert scheme of points on a surface. This specialization was already considered in \cite[Section 6]{virasorotoricPT}.

Let $S$ be a non-singular and projective surface such that $H^1(S)=0$. We denote by $S^{[n]}$ the Hilbert scheme of points on $S$ parametrizing 0 dimensional subschemes of $S$ with length $n$. If we set $X=S\times \PP^1$ and $\beta=n[\PP^1]$, the minimal Euler characteristic of a stable pair in $X=S\times \PP^1$ with support in a curve of class $\beta=n[\PP^1]$ is $n$ and we have an isomorphism of schemes
\[P_n(S\times \PP^1, n[\PP^1])\cong S^{[n]}.\]
The isomorphism is defined by sending $\xi\in S^{[n]}$ to the stable pair \[\O_{S\times \PP^1}\to \O_{\xi\times \PP^1}.\]
Moreover, since $S^{[n]}$ is smooth and has the expected dimension
\[2n=\int_{n[\PP^1]}c_1(S\times \PP^1)\]
the virtual fundamental class of $P_n(S, n[\PP^1])$ is just the fundamental class. 

We define the algebra of descendents and the geometric realization of descendents in $S^{[n]}$ parallel to the stable pairs definitions.

\begin{definition}
Given a surface $S$, we let $\DD^S$ be the commutative algebra generated by
\[\{\ch_k(\gamma): k\geq 0, \gamma\in H^\ast(S)\}\]
subject to the linearity relations.
\end{definition}

\begin{definition}
Given a surface $S$ and $n\geq 0$, we denote by $\Sigma_n\subseteq S^{[n]}\times S$ the universal subscheme and we let $\pi_1: S^{[n]}\times S\to S^{[n]}$ and $\pi_2: S^{[n]}\times S\to S$ denote the projections onto the two factors. 

Given $k\in \ZZ_{\geq 0}$ and $\gamma\in H^{\ast}(S)$, we define the geometric realization of descendents by
$$\ch_k(\gamma)=(\pi_2)_\ast
\left(\ch_k\left(\O_{\Sigma_n}-\O_{S^{[n]}\times S}\right)\cdot \pi_1^\ast(\gamma) \right)\in H^\ast(S^{[n]})$$
\end{definition}

The stable pairs descendents in $P_n(S\times \PP^1, n[\PP^1])$ are determined by the Hilbert scheme descendents:
\begin{align*}
\ch_k^{\textup{PT}}(\gamma\times 1)=0\textup{ and }
\ch_k^{\textup{PT}}(\gamma\times \p)=\ch_k^{\textup{Hilb}}(\gamma)
\end{align*}
where $1, \p\in H^\ast(\PP^1)$ are the unit and point classes. In particular, the Virasoro constraints on $P_n(S\times\PP^1, n[\PP^1])$ are equivalent to constraints on integrals of descendents on $S^{[n]}$. We formulate these constraints now:

\begin{itemize}
\item For $k\geq -1$, define a derivation $R_k$ on $\DD^S$ by fixing its action on the generators: given $\gamma\in H^{p,q}(S)$, let
\begin{equation}
R_k(\ch_i(\gamma))=\left(\prod_{j=0}^k(i+p-2+j)\right)\ch_{i+k}(\gamma).
\end{equation}
\item The operator $T_k: \DD^S\to \DD^S$ is multiplication by a fixed element of $\DD^S$:
\begin{align}T_k&=\sum_{a+b=k+2}(-1)^{p^Lp^R}(a+p^L-2)!(b+p^R-2)!\ch_a\ch_b(1)\nonumber \\
&+\sum_{a+b=k}a!b!\ch_a\ch_b\left(\frac{c_1^2+c_2}{12}\right).\end{align}
We are using the abbreviation 
\[(-1)^{p^Lp^R}(a+p^L-2)!(b+p^R-2)!\ch_a\ch_b(1)\]
for
\[\sum_{i}(-1)^{p^L_ip^R_i}(a+p^L_i-2)!(b+p^R_i-2)!\ch_a(\gamma^L_i)\ch_b(\gamma^R_i)\]
where $\sum_{i}\gamma^L_i\otimes \gamma^R_i$ is the Kunneth decomposition of the diagonal class $\Delta\in H^\ast(S\times S)$ and $\gamma^L_i$, $\gamma^R_i$ have Hodge types $(p^L_i, q^L_i)$ and $(p^R_i, q^R_i)$, respectively.
\item For $\alpha \in H^\ast(S)$, we define the derivation $R_{-1}[\alpha]: \DD^S\to \DD^S$ by its action on the generators:
\[R_{-1}[\alpha](\ch_i(\gamma))=\ch_{i-1}(\alpha\gamma).\]
In particular, $R_{-1}[1]=R_{-1}$. For $k\geq -1$, we define the operators $S_k: \DD^S\to \DD^S$ by
\[S_k=(k+1)!\sum_{p_i^L=0} R_{-1}[\gamma_i^L]\ch_{k+1}(\gamma^R_i).\]
The sum runs over the terms $\gamma_i^L\otimes \gamma_i^R$ of the Kunneth decomposition of $\Delta\in H^\ast(S\times S)$ such that $p_i^L=0$. 
\end{itemize}

\begin{definition}
We define the operators $\L_k: \DD^S\to \DD^S$, for $k\geq -1$, by
\begin{equation}\L_k=\L_k^S=R_k+T_k+S_k.
\end{equation}
\end{definition}

One of the two main results of this paper is that indeed these operators impose universal constraints on descendent integrals on the Hilbert scheme of points of $S$, as predicted by the Virasoro conjecture for $P_n(S\times \PP^1, n[\PP^1])$.

\begin{theorem}
\label{main}
Let $S$ be a surface with $H^1(S)=0$ and let $D\in \DD^S$. Then
\[\int_{S^{[n]}} \L_k D=0.\]
\end{theorem}

This result, in the case that $S$ is a toric surface, follows from the Virasoro constraints for the stable pairs theory of toric 3-folds, \cite[Section 6]{virasorotoricPT}.

\begin{theorem}[Theorem 20 in \cite{virasorotoricPT}]
\label{maintoric}
Theorem \ref{main} holds when $S$ is a (connected) toric surface.
\end{theorem}

\subsection{Plan of the paper}

\noindent In section \ref{sec: intertwining} we explain how to adapt some arguments of \cite{virasorotoricPT} to get a more general version of theorem \ref{virGWimpliesvirPT} that allows (simply-connected) 3-folds with non-$(p,p)$-cohomology as long as we don't have insertions with $(0,p)$ classes. In particular, we explain the appearance of $S_k$ in proposition \ref{intertwiningSk}. This section is completely independent of the rest of the paper.

The two main results of this paper are verifications of conjecture \ref{conj: virasoro} in two instances:

\begin{theorem}
\label{mainsurface}
Let $S$ be a surface with $H^1(S)=0$ and $n\in \ZZ_{\geq 0}$. Conjecture \ref{conj: virasoro} holds when $X=S\times \PP^1$, $\beta=n[\PP^1]$ and the Euler characteristic is $n$, i.e.
\[\left\langle \L_k(D)\right\rangle_{n, n[\PP^1]}^{S\times \PP^1, \textup{PT}}=0.\]
\end{theorem}

\begin{theorem}
\label{maincubic}
Conjecture \ref{conj: virasoro} holds when $X$ is a cubic 3-fold and $\beta\in H_2(X; \ZZ)$ is the line class. 
\end{theorem}

The natures of the two proofs are quite different. For the surface, theorem \ref{mainsurface} is formally equivalent to the Virasoro constraints for descendents in the Hilbert scheme of points on a surface, theorem \ref{main}. We give the proof of theorem \ref{main} in section \ref{sec:surface}. The basic idea of the proof is to reduce the general case to the toric case via the existence of universal  formulas for integration of descendents on Hilbert schemes of points, from \cite{egl}. Two interesting aspects of the proof are the need to allow disconnected surfaces and the role that the Hodge degrees play.

For the cubic 3-fold, we compute all the stable pairs invariants and then we check directly that the Virasoro constraints hold by verifying some identities. The computation of all the invariants is done in sections \ref{cubic3foldn1} (case $n=1$) and \ref{cubic3foldngeneral} (case $n>1$). The verification of the Virasoro constraints is done in section \ref{sec: cubic}. The computation of the invariants is as directly from the definitions as possible: we identify the moduli spaces (which are smooth), the virtual fundamental class and compute expressions for all the descendents.

\subsection*{Acknowledgements}
The author would like to thank his advisor R. Pandharipande for many useful discussions and for suggesting both the problems treated in this paper. Discussions with A. Oblomkov regarding the PT Virasoro operators and the PT/GW transformation were also extremely useful. 

This project has received funding from the European Research Council (ERC)
under the European Unions Horizon 2020 research and innovation programme
(ERC-2017-AdG-786580-MACI).

\section{Intertwining}
\label{sec: intertwining}

\noindent In this section we will very briefly explain what can be recovered of the intertwining in \cite{virasorotoricPT} between the Gromov-Witten Virasoro operators and the stable pairs Virasoro operators. This section is highly dependent on \cite{virasorotoricPT} and we'll use the notation from there. In particular, the reader should be aware of the definition of the stationary GW/PT correspondence, section 0.6, and the key intertwining statement, theorem 12.

\subsection{Intertwining between $R_k+T_k$ and $\tilde L_k^{\textup{\normalfont GW}}$}

\noindent The proof of theorem 12 in \cite{virasorotoricPT} can be entirely adapted to our more general situation to show that
\[\mathfrak{C}^\bullet \circ L_k^{\textup{PT}}(D)= (\iota u)^{-k}\tilde L_k^{\textup{GW}}\circ \mathfrak{C}^\bullet (D)\]
for any $D$ in the algebra $\DD^{X, p>0}_{\textup{PT}}$ generated by descendents $\ch_i(\gamma)$ where $\gamma\in H^{p,q}(X)$ for $p>0$. 

The necessary modifications for the proof of \cite[Theorem 12]{virasorotoricPT} are basically replacing every (complex) cohomological degree by the first Hodge degree. In particular, in the statements and proofs of propositions 16, 17 and 18 we replace every condition ``$\alpha\in H^{2p}(X)$'' by ``$\alpha\in H^{p, \ast}(X)$''. 

We note that descendents with $p=0$ are not treated in \cite{virasorotoricPT} and that's why we exclude them now. Indeed it seems that the key propositions in the proof of theorem 12, namely propositions 16, 17, 18 and 19, fail when we allow descendents of $(0,p)$ classes.

\subsection{Intertwining between $S_k$ and $T_k^0$}
\noindent The main difference between the Virasoro operators in definition \ref{def: virasoroop} and their specialization to toric varieties, in \cite[Definition 2]{virasorotoricPT}, is the operator $S_k$. In the toric case, $S_k$ specializes to $(k+1)!R_{-1}\ch_{k+1}(\p)$. This term is obtained by applying the transformation to the operator $T_k^0$ on the Gromov-Witten side, \cite[Equation (18)]{virasorotoricPT}. In the general case, $T_k^0$ is transformed via the GW/PT into $S_k$. Indeed we have

\[\frac{1}{2}T_k^0=(k+1)!\sum_{p_i^L=0}\colon\tau_0(\gamma_i^L)\tau_{k-1}(\gamma_i^R)\colon\]
where the sum runs over the terms $\gamma_i^L\otimes \gamma_i^R$ of the Kunneth decomposition of $\Delta\in H^\ast(X\times X)$ with $p_i^L=0$.

Given a class $\alpha\in H^\ast(X)$, we define the following operators on the Gromov-Witten side twisted by $\alpha$:
\begin{itemize}
\item We let $R_{-1}^{\textup{GW}}[\alpha]: \DD^X_{\textup{GW}}\to \DD^X_{\textup{GW}}$ be a derivation defined on the generators of the algebra by
\[R_{-1}^{\textup{GW}}[\alpha](\tau_j(\gamma))=\tau_{j-1}(\alpha\gamma).\]
\item We define a quadratic differential operator $B_0[\alpha]: \DD^X_{\textup{GW}}\to \DD^X_{\textup{GW}}$ by fixing its action on products of two generators:
\[B_0[\alpha](\ch_i(\gamma)\ch_j(\gamma'))=\delta_{i}\delta_j\int_{X}\alpha \gamma\gamma'.\]
Here $\delta_i$ represents the usual Kronecker delta, giving 1 if $i=0$ and 0 otherwise.
\item Finally, we define $L_{-1}^{\textup{GW}}[\alpha]: \DD^X_{\textup{GW}}\to \DD^X_{\textup{GW}}$ by
\[L^{\textup{GW}}_{-1}[\alpha]=\tau_0(\alpha)-R_{-1}^{\textup{GW}}[\alpha]+\frac{(\iota u)^2}{2}B^{0}[\alpha].\]
\end{itemize}

If $\alpha\in H^{0, q}(X)$, the operators $L_{-1}^{\textup{GW}}[\alpha]$ constrain the Gromov-Witten invariants, that is, we have
\[\left \langle L_{-1}^{\textup{GW}}[\alpha](D) \right \rangle^{X}_\beta=0.\]
When $\alpha=1$ this is just the usual string equation, see \cite{witten}. The general case can be shown with a minor modification of the usual proof of the string equation; the condition that $\alpha\in H^{0,q}(X)$ ensures a vanishing corresponding to \cite[Equation 2.38]{witten} in the case $\alpha=1$. 

To state the next proposition we introduce the following variation of the $S_k$ operator:
\[\tilde S_k=\sum_{p_i^L}R_{-1}[\gamma_i^L]\ch_{k+1}(\gamma_i^R)-(-1)!\ch_1(\gamma^L_i c_1)\ch_{k+1}(\gamma^R_i).\]
The non-geometric descendents $(-1)!\ch_1(\gamma)$ play an important role in the intertwining property of \cite{virasorotoricPT} and are explained there.

\begin{proposition}
\label{intertwiningSk}
For any $D\in \DD^{X, p>0}_{\textup{PT}}$, we have
$$\left\langle \mathfrak{C}^\bullet(\tilde S_k D)\right\rangle^{X, \textup{GW}}_\beta
=\frac{(\iota u)^{2-k}}{2}\left\langle T_k^0\mathfrak{C}^\bullet (D)\right\rangle^{X, \textup{GW}}_\beta.$$
\end{proposition}
\begin{proof}
The proof of \cite[Proposition 11]{virasorotoricPT} is easily adapted to show that we have, for any $\alpha\in H^{0,q}(X)$, the following intertwining:
\[\mathfrak{C}^\bullet\circ \left(R_{-1}^{\textup{PT}}[\alpha]-(-1)!\ch_1(c_1\alpha)\right)=(\iota u)\left(R_{-1}^{\textup{GW}}[\alpha]-\frac{(\iota u)^2}{2}B^0[\alpha]\right)\circ\mathfrak{C}^\bullet.\]

Consider now a term $\gamma^L\otimes \gamma^R$ of the Kunneth decomposition with $\gamma^L\in H^{0, q}(X)$ and $\gamma^R\in H^{3, 3-q}(X)$. Since $D\in \DD^{X, p>0}_{\textup{PT}}$, the descendent $\ch_{k+1}(\gamma^R)$ doesn't bump $D$, hence
\[\mathfrak{C}^\bullet(\ch_{k+1}(\gamma^R)D)=\mathfrak{C}^\circ(\ch_{k+1}(\gamma^R))\mathfrak{C}^\bullet(D)=(\iota u)^{-k+1}\tau_{k-1}(\gamma^R)\mathfrak{C}^\bullet(D).\]
Combining the previous observations with the $\gamma^L$-string equation we get
\begin{align*}
\left\langle \tau_0(\gamma^L)\tau_{k-1}(\gamma^R)\mathfrak C^\bullet(D)\right\rangle&^{\textup{GW}}_\beta
=(\iota u)^{k-1}\left\langle \tau_0(\gamma^L)\mathfrak C^\bullet(\ch_{k+1}(\gamma^R)D)\right\rangle^{\textup{GW}}_\beta\\
&=(\iota u)^{k-1}\left\langle \left(R_{-1}^{\textup{GW}}[\gamma^L]-\frac{(\iota u)^2}{2}B^0[\gamma^L]\right)\mathfrak C^\bullet(\ch_{k+1}(\gamma^R)D)\right\rangle^{\textup{GW}}_\beta\\
&=(\iota u)^{k-2}\left\langle
\mathfrak C^\bullet \left( R_{-1}^{\textup{PT}}[\gamma^L]\ch_{k+1}(\gamma^R)D\right.\right.\\
&\left.\left.\quad \quad -(-1)!\ch_1(\gamma^L c_1)\ch_{k+1}(\gamma^R)D\right)\right
\rangle^{\textup{GW}}_\beta
\end{align*}

The proposition then follows by summing over the terms $\gamma_i^L\otimes \gamma_i^R$ of the Kunneth decomposition with $p_i^L=0$.
\end{proof}

\subsection{Extending theorem \ref{virGWimpliesvirPT}}

\noindent The previous discussion provides the adaptations needed to extend theorem \ref{virGWimpliesvirPT} of \cite{virasorotoricPT}. The original result says that for a 3-fold with only $(p,p)$-cohomology (as is the case of toric 3-folds) the stationary Gromov-Witten Virasoro combined with the stationary GW/PT correspondence imply the stationary stable pairs Virasoro. We extend this to any simply-connected 3-fold and to the algebra $\DD^{X, p>0}_{\textup{PT}}$.

\begin{theorem}
\label{virGWimpliesvirPT2}
Let \(X\) be a projective smooth simply-connected 3-fold for which 
the following two properties are satisfied:
\begin{enumerate}
\item[(i)]
The stationary Virasoro constraints for the Gromov-Witten theory of $X$ hold.
\item[(ii)]
The stationary $\textup{GW}/\textup{PT}$ correspondence holds (see \cite[Section 0.6]{virasorotoricPT}). 
\end{enumerate}
Then for any $k\geq -1$, $n\in \ZZ$, $\beta\in H_2(X; \ZZ)$ and $D\in \DD^{X, p>0}_{\textup{PT}}$ we have
\[\langle \L_k(D)\rangle_{n, \beta}^{X, \textup{PT}}=0.\]
\end{theorem}

Although the stationary GW/PT correspondence allows descendents in the larger stationary algebra $\DD^{X+}_{\textup{PT}}\supseteq \DD^{X, p>0}_{\textup{PT}}$, we were not able to prove an adequate intertwining statement in the presence of descendents of $(0,q)$-classes. Note that if conjecture \ref{conj: vanishing} holds then we automatically have $\langle \L_k(D)\rangle_{n, \beta}^{X, \textup{PT}}=0$ when $D$ contains descendents of $(0,2)$ or $(0,3)$ classes since all the terms in the expansion of $\langle \L_k(D)\rangle_{n, \beta}^{X, \textup{PT}}$ will also contain such descendents.

In the two examples discussed in the paper, the cubic 3-fold and $S\times \PP^1$ with $\beta=n[\PP^1]$, these issues don't exist. In the cubic 3-fold $H^{0,2}(X)=H^{0,3}(X)=0$ and in the surface case we have $\ch_k(\gamma)=0$ for $\gamma\in H^{0,q}(X)$.

\section{Virasoro constraints for the Hilbert scheme of points of a surface}
\label{sec:surface}

\noindent In this section we will give the proof of theorem \ref{main}. The key idea is to use the existence of universal formulas for integrals of descendents in the Hilbert scheme, in the spirit of \cite{egl}, to reduce to the toric case, theorem \ref{maintoric}. The analysis of the universal formulas and their interaction with the Virasoro operators is done in \ref{subsec: universal}. Proposition \ref{zariskidensity} is the ingredient to show that disconnected toric surfaces provide enough data to show the vanishing in general; in \ref{subsec: disconnected} we explain how to deal with disconnected surfaces. Finally, the actual argument for the proof of theorem \ref{main} is given in subsections \ref{subsec: mainhodge} and \ref{subsec: mainnonhodge}. The first treats the case where $D\in \DD^S$ only contains descendents of $(p,p)$ classes and follows almost immediately from the previous steps. In subsection \ref{subsec: mainnonhodge} we allow non-$(p,p)$ insertions. The key trick here is to replace the $(0,2)$ and $(2,0)$ insertions by $(0,0)$ and $(2,2)$ insertions, respectively. For this we need once again to consider additional connected components.

\subsection{Disconnected surfaces}
\label{subsec: disconnected}

\noindent Suppose that $S$ is a disconnected surface and admits a decomposition $S=S_1\sqcup S_2$. We'll describe the descendents of $S$ in terms of descendents of $S_1$ and $S_2$ and we'll conclude that if theorem \ref{main} holds for $S_1$ and $S_2$ then it also holds for $S$. 

The cohomology of $S$ is the direct sum
\[H^\ast(S)=H^\ast(S_1)\oplus H^\ast(S_2).\]
If $(\gamma_1, \gamma_2)\in  H^\ast(S_1)\oplus H^\ast(S_1)$ we'll denote the corresponding class by $\gamma_1+\gamma_2\in H^\ast(S)$. We have 
\[\DD^{S}=\DD^{S_1}\otimes \DD^{S_2}\]
and, given $D_i\in \DD^{S_i}$, we denote by $D_1\otimes D_2\in \DD^S$ the corresponding element.

\begin{lemma}
If $S=S_1\sqcup S_2$ then \[\L_k^S: \DD^{S_1}\otimes \DD^{S_2}=\DD^S\to \DD^S = \DD^{S_1}\otimes \DD^{S_2}\]
is given by 
\[\L_k^S=\textup{id}_{\DD^{S_1}}\otimes \L_k^{S_2}+\L_k^{S_1}\otimes \textup{id}_{\DD^{S_2}}.\]
\end{lemma}
\begin{proof}
This property holds for the 3 operators $R_k, T_k, S_k$ defining $\L_k$. For $R_k$ it holds simply because $R_k$ is a derivation. For both $T_k$ and $S_k$ it holds since the diagonal class of $S$ is the sum of the diagonal classes of $S_1$ and $S_2$ via the inclusions $H^\ast(S_i\times S_i)\hookrightarrow H^\ast(S\times S)$. For the $S_k$ operator note also that for $\gamma_i^L\in H^\ast(S_1)\hookrightarrow H^\ast(S)$ we have
\[R_{-1}^S[\gamma_i^L](D_1\otimes D_2)=(R_{-1}^{S_1}[\gamma_i^L]D_1)\otimes D_2\]
since $R_{-1}^S[\gamma_i^L]$ is a derivation and it doesn't interact with descendents coming from $S_2$. 
\end{proof}

We now describe the evaluation map $\langle \cdot \rangle^S: \DD^S\to \QQ$ in terms of the evaluation maps of $S_1$ and $S_2$. 

\begin{proposition}
Let $S=S_1\sqcup S_2$, let $D_1\otimes D_2\in \DD^S=\DD^{S^1}\otimes \DD^{S^2}$ and let $n\geq 0$. Then 
\[\langle D_1\otimes D_2\rangle^S_n=\sum_{n_1+n_2=n}\langle D_1\rangle^{S_1}_{n_1}\langle D_2\rangle^{S_2}_{n_2}\]
where the sum runs over $n_1, n_2\geq 0$ summing to $n$. Hence
\begin{equation}\label{virasorodisconnected}\langle \L_k^S(D_1\otimes D_2)\rangle^S_n=\sum_{n_1+n_2=n}\left(\langle D_1\rangle^{S_1}_{n_1}\langle \L_k^{S_2}(D_2)\rangle^{S_2}_{n_2}+\langle \L_k^{S_1}(D_1)\rangle^{S_1}_{n_1}\langle D_2\rangle^{S_2}_{n_2}\right)
\end{equation}
\end{proposition}

\begin{proof}
We begin with the decomposition
\[S^{[n]}=\bigsqcup_{n_1+n_2=n}S_1^{[n_1]}\times S_2^{[n_2]}.\]
The universal subscheme $\Sigma_n^S$ (we use the superscript $S$ to make the surface we're referring to explicit) admits a decomposition in connected components
\[\Sigma_n^S=\bigsqcup_{n_1+n_2=n}\Sigma_{n_1}^{S_1}\times S_2^{[n_2]}\sqcup \bigsqcup_{n_1+n_2=n}S_1^{[n_1]}\times \Sigma_{n_2}^{S_2}.\]
Here $\Sigma_{n_1}^{S_1}\times S_2^{[n_2]}$ is contained in the connected component $S_1\times S_1^{[n_1]}\times S_2^{[n_2]}$ of $S\times S^{[n]}$. Thus
\[\ch_k\left(\O_{\Sigma_n^S}\right)=\sum_{n_1+n_2=n}\left(1_{S_1^{[n_1]}}\otimes \ch_k\left(\O_{\Sigma_{n_2}^{S_2}}\right)+\ch_k\left(\O_{\Sigma_{n_1}^{S_1}}\right)\otimes 1_{S_2^{[n_2]}}\right).\]

It follows formally that the geometric realization of $D_1\otimes D_2$ in 
\[H^\ast(S^{[n]})=\bigoplus_{n_1+n_2=n}H^\ast(S_1^{[n_1]}\times S_2^{[n_2]})\]
is given in each component of the direct sums by the external product of the geometric realizations of $D_1$ and $D_2$ in $H^\ast(S_1^{[n_1]})$ and $H^\ast(S_2^{[n_2]})$, respectively. The result then follows.
\end{proof}

From equation \eqref{virasorodisconnected} the following holds:

\begin{corollary}
\label{disconnected}
Let $S=S_1\sqcup S_2$ be a disconnected surface. If theorem \ref{main} holds for $S_1$ and $S_2$ then it also holds for $S$.  
\end{corollary}

Combining corollary \ref{disconnected} with the result for toric surfaces, theorem \ref{maintoric}, it follows that the Virasoro constraints also hold for disconnected toric surfaces.

\begin{corollary}
\label{disconnectedtoric}
Theorem \ref{main} holds for disconnected toric surfaces.
\end{corollary}

\subsection{Universal formulas for integrals of descendents on the Hilbert scheme}
\label{subsec: universal}

\noindent Integrals of descendents on the Hilbert scheme of points on a (possibly disconected) surface admit universal expressions by a well known argument due to Ellingsrud, Göttsche and Lehn, \cite{egl}. The original result of \cite{egl} is for descendents of $K$-theory classes or, equivalently, descendents of cohomology classes in the image of the map $K(S)\to H^\ast(S)$ mapping $\alpha\in K(X)$ to $\ch(\alpha)\td(X)$. The recursive argument of \cite{egl} was adapted to our setting in \cite{lqw}. These universal formulas are polynomials in integrals involving insertions and the Chern classes $c_1=c_1(TS)$ and $c_2=c_2(TS)$.
\begin{definition}
Let $(\gamma_1, \ldots, \gamma_m)$ be a $m$-tuple of classes in $H^\ast(S)$. The set of integrals of $(\gamma_1, \ldots, \gamma_m)$ is the assignment
\[(I, \varepsilon)\mapsto P_I^\varepsilon=\int_{S}c_1^{\varepsilon^1}c_2^{\varepsilon^2}\prod_{j\in I}\gamma_j\]
\footnote{The $\gamma_j$ in the product $\prod_{j\in I}\gamma_j$ are ordered according to the natural ordering of $I\subseteq \{1, \ldots, m\}$. We allow the set $I$ to be empty.}where $I\subseteq\{1, \ldots m\}$ and $\varepsilon=(\varepsilon^1, \varepsilon^2)\in \{(0,0), (1, 0), (2,0), (0, 1)\}$. If $D=\prod_{j=1}^m\ch_{k_j}(\gamma_j)\in \DD^S$ we say that the set of integrals of $D$ is the set of integrals of $(\gamma_1, \ldots, \gamma_m)$.
\end{definition}

In particular, taking $I=\emptyset$, the numbers $\int_S c_1^2$ and $\int_S c_2$, which depend only on $S$, are part of the set of integrals of any $D$.

\begin{theorem}[(\cite{egl}, \cite{lqw})]
\label{universal}
Given fixed $k_1, \ldots, k_m$ and $n$, let $D=\prod_{j=1}^m\ch_{k_j}(\gamma_j)$. Then the evaluation $\langle D\rangle_n^S$ is a polynomial in the integrals of $D$. More precisely,
\[\langle D\rangle^S_n=\sum_{\varepsilon_1, \ldots, \varepsilon_k, I_1, \ldots, I_k}a_{I_1, \ldots I_k}^{\varepsilon_1, \ldots, \varepsilon_k}\prod_{j=1}^k P_{I_j}^{\varepsilon_j}\]
where the sum runs over every partition (possibly with empty parts)
\[[m]=\bigsqcup_{j=1}^k I_j\]
and every possibility of $\varepsilon_j$. The coefficients $a_{I_1, \ldots I_k}^{\varepsilon_1, \ldots, \varepsilon_k}\in \QQ$ depend only on $k_1, \ldots, k_m, n$. 
\end{theorem}

Our formulation uses the fact that $\langle D\rangle_n^S$ depends linearly on the classes $\gamma_j$. In particular this explains why there are no integrals involving higher powers of the classes $\gamma_j$ and why the sets $I_j$ form a partition.

From this we can get similar universal formulas for $\langle \L_k D\rangle^S_n$. To prove those, we'll use the following identities which are easy consequences of Hirzebruch-Riemann-Roch.

\begin{lemma}\label{lem: identitieshrr}For any surface $S$ we have the identities
\[\sum_{p^L_j=0}\gamma^L_j\gamma^R_j=\frac{1}{12}\left(c_1^2+c_2\right)=\sum_{p^L_j=2}\gamma^L_j\gamma^R_j \textup{ and }\sum_{p^L_j=1}\gamma^L_j\gamma^R_j=\frac{1}{6}\left(5c_2-c_1^2\right)\]
where the sums run over the terms in the Kunneth decomposition $\sum_j\gamma_j^L\otimes \gamma_j^R$ of the diagonal with $p_j^L=0,1,2$.
\end{lemma}
\begin{proof}
Writing the class of the diagonal using some dual basis of $H^\ast(S)$ respecting the Hodge grading one can show that
\[\int_S\sum_{p_j^L=k}\gamma_j^L\gamma_j^R=(-1)^k\left(h^{k,0}-h^{k,1}+h^{k,2}\right).\]
By Hirzebruch-Riemann-Roch the latter may be expressed using the Chern classes 
\[h^{k,0}-h^{k,1}+h^{k,2}=\chi(\Omega_S^k)=\int_{S} \ch(\Omega_S^k)\td(S)=\begin{cases}
\int_S \td(S)=\int_{S}\frac{1}{12}(c_1^2+c_2)&\textup{ if }k=0,2\\
\int_S\frac{1}{6}\left(5c_2-c_1^2\right)&\textup{ if }k=1.
\end{cases}\]
\end{proof}

Note that summing over $k$ produces the well-known identity $\Delta^\ast \Delta_\ast 1=c_2(S)=\chi^{\textup{top}}(S)$. Ultimately, the last lemma (more precisely: the fact we can write such sums in terms of Chern classes) is the crucial property of the Hodge degrees in our proof. The fact that we can't express sums over terms in the diagonal with fixed cohomological degree explains why our proof wouldn't work if we had defined the Virasoro operators using the cohomological degree.

\begin{lemma}
\label{universalvirasorohodge}
Let 
\[D=\prod_{i=1}^m\ch_{k_i}(\gamma_i)\in \DD^S\]
where $\gamma_i\in H^{p_i, q_i}(S)$. Assume that $k,m, k_i, p_i$ are all fixed. Then
\[\langle \L_k D\rangle^S_n\]
is a polynomial in the set of integrals of $(\gamma_1, \ldots, \gamma_s)$.
\end{lemma}
\begin{proof}
The claim for the term
\[\langle R_kD \rangle^S_n\] 
follows immediately from theorem \ref{universal}. It remains to study the operators $T_k, S_k$. Given a term $\gamma\otimes \gamma'$ in the Kunneth decomposition of the diagonal, we consider the formal expression provided again by theorem \ref{universal}:

\begin{align}\label{diagonaltermsuniv}\nonumber \langle\ch_a(\gamma)&\ch_b(\gamma')D \rangle^S_n=\sum_{I}\left(\int_S \gamma \gamma' \prod_{i\in I}\gamma_i\right)X_{a,b}^I\\
&+\sum_{J_1, J_2, \varepsilon_1, \varepsilon_2}\left(\int_S \gamma \prod_{j\in J_1}\gamma_jc_1^{\varepsilon_1^1}c_1^{\varepsilon_1^2}\right)\left(\int_S \gamma' \prod_{j\in J_2}\gamma_jc_1^{\varepsilon_2^1}c_1^{\varepsilon_2^2}\right)Y_{a,b}^{J_1, J_2, \varepsilon_1, \varepsilon_2}
\end{align}
where the first runs through every subset $I\subseteq [m]$ and the second sum runs through disjoint subsets $J_1, J_2\subseteq [m]$ and $\varepsilon_1, \varepsilon_2\in \{(0,0), (1,0), (2,0), (0,1)\}$. Both $X_{a,b}^{I}$ and $Y_{a,b}^{J_1, J_2, \varepsilon_1, \varepsilon_2}$ are expressions depending only on the polynomials of $D$, on $a,b$ and on the fixed variables.

We look to the contribution of the two lines of \eqref{diagonaltermsuniv} to the diagonal part of $\langle T_k D\rangle^S_n$ given by
\begin{equation}
\label{tkdiagonal}
\sum_{j}(-1)^{p^L_j p^R_j}(a+p^L_j-2)!(b+p^R_j-2)!\langle\ch_a(\gamma_j^L)\ch_b(\gamma_j^R) D\rangle^S_n.
\end{equation}
We begin with the second line. First we note that the terms in the second line of the expansion \eqref{diagonaltermsuniv} of $\langle\ch_a(\gamma_j^L)\ch_b(\gamma_j^R) D\rangle^S_n$ vanish unless we have
\[p_j^L=p_{J_1, \varepsilon_1}\equiv 2-\sum_{i\in J_1} p_i-\varepsilon_1^1-2\varepsilon_1^2 \textup{ and } p_j^R=p_{J_2, \varepsilon_2}\equiv 2-\sum_{i\in J_2} p_i-\varepsilon_2^1-2\varepsilon_2^2.\]
Hence the contribution of the second line of \eqref{diagonaltermsuniv} to \eqref{tkdiagonal} is
\[\sum_{J_1, J_2, \varepsilon_1, \varepsilon_2}(-1)^{p_{J_1, \varepsilon_1}p_{J_2, \varepsilon_2}}(a+p_{J_1, \varepsilon_1}-2)!(b+p_{J_2, \varepsilon_2}-2)!P^{\varepsilon_1+\varepsilon_2}_{J_1\sqcup J_2}Y_{a,b}^{J_1, J_2, \varepsilon_1, \varepsilon_2}\]
and $p_{J_1, \varepsilon_1},p_{J_2, \varepsilon_2}$ are determined by the collection of numbers $p_i$ as above, so the claim is proven for this term. 

For the first line of \eqref{diagonaltermsuniv} we used the identities in lemma \ref{lem: identitieshrr}. The contribution of the first line of \eqref{diagonaltermsuniv} to \eqref{tkdiagonal} is
\begin{align*}\sum_{I}&\left(\frac{1}{12}(a-2)!b!\int_S \left(\left(c_1^2+c_2\right)\prod_{i\in I}\gamma_i\right)+\frac{1}{12}a!(b-2)!\int_S \left(\left(c_1^2+c_2\right)\prod_{i\in I}\gamma_i\right)
\right.\\
&\left.-\frac{1}{6}(a-1)!(b-1)!\int_S \left(\left(-c_1^2+5c_2\right)\prod_{i\in I}\gamma_i\right)\right)X_{a,b}^I
\end{align*}
The argument for $\ch_a\ch_b\left(c_1^2+c_2\right)$ is analogous and easier.

For $S_k$ a similar analysis is once again possible. Theorem \ref{universal} provides again a universal expression for
\[\langle R_{-1}[\gamma]\ch_{k+1}(\gamma')D\rangle_n^S\]
similar to the one in the right hand side of \eqref{diagonaltermsuniv}. The rest of the argument is similar. The contribution of (the analogue of) the first line of \eqref{diagonaltermsuniv} is treated exactly in the same way using $\sum_{p^L_j=0}\gamma^L_j\gamma^R_j=\frac{1}{12}\left(c_1^2+c_2\right)$. The contribution of (the analogue of) the second line is again a sum of $P_{J_1\sqcup J_2}^{\varepsilon_1+\varepsilon_2}$, with certain coefficients, running over $J_1, J_2, \varepsilon_1, \varepsilon_2$ such that $p_{J_1, \varepsilon_1}=0, p_{J_2, \varepsilon_2}=3$. Since the numbers $p_{J_i, \varepsilon_i}$ are determined by the collection of numbers $p_i$, once again we get a similar universal expression for $\langle S_k D\rangle_n^S$ as a polynomial in the integrals of $D$.
\end{proof}

\subsection{Zariski density of data from toric varieties}
\label{subsec: zariski}

\noindent The next proposition will show that disconnected toric varieties provide enough data points to guarantee that the Virasoro constraints always hold. Taking disconnected surfaces is necessary since for a connected toric surface Hirzebruch-Riemann-Roch gives the restriction on the data
\[\int_S c_1^2+\int_S c_2=12\chi^{\textup{hol}}(S)=12.\]

\begin{proposition}
\label{zariskidensity}
Fix $m\geq 0$. Given a (possibly disconnected) toric surface $S$ and classes $\gamma_1, \ldots, \gamma_m\in H^2(S)$, we associate to this data a $\left(\binom{m+1}{2}+m+2\right)$-tuple of rational numbers
$$\left\{\int_S \gamma_i\gamma_j\right\}_{1\leq i\leq j\leq m}
\cup \left\{\int_S\gamma_i c_1\right\}_{1\leq i\leq m}
\cup \left\{\int_S c_1^2, \int_S c_2\right\}.$$
By varying the toric surface and the classes $\gamma_j$, the set of possible such $\left(\binom{m+1}{2}+m+2\right)$-tuples is Zariski dense in $\QQ^{\binom{m+1}{2}+m+2}$. 
\end{proposition}

\begin{proof}
We start with the union of $N\geq 2$ copies of $\PP^1\times \PP^1$. Picking one of the copies, we successively  perform $M$ toric blow-ups at points fixed by the torus action; we call $S$ the resulting disconnected surface. We do so in a way that the last $m$ blow-ups have disjoint exceptional divisors $D_1, \ldots, D_m$; this is possible as long as $M$ is large enough, namely $M\geq \max\{m,2m-4\}$ (for example, if $m=4$ we just blow-up the $4$ vertices of $\PP^1\times \PP^1$). Let $D_0$ be a divisor $[\p\times \PP^1]$ in another copy of $\PP^1\times \PP^1$. Let
\[\gamma_i=\sum_{j=0}^m a_{ij}D_j\]
with $a_{ij}\in \QQ$ for $1\leq i\leq m, 0\leq j\leq m$. 

One checks immediately that we have
\[ \int_S c_1(S)^2=8N-M\textup{ and }\int_S c_2(S)=4N+M\]
since blowing up one point increases the integral $\int_S c_2(S)$ by 1 and decreases $\int_S c_1(S)^2$ by 1.
 
The set of pairs
\[\{(8N-M,4N+M): N\geq 2, M\geq \max\{m,2m-4\}\}\subseteq \QQ^2\]
is Zariski dense in $\QQ^2$, so it's enough to show that fixing $M, N$ and varying $a_{ij}$ produces a Zariski dense set of $\left(\binom{m+1}{2}+m\right)$-tuples
\[\left\{\int_S\gamma_i\gamma_j\right\}_{1\leq i\leq j\leq m}\cup \left\{\int_S\gamma_ic_1\right\}_{1\leq i\leq m}.\]
By construction of the divisors $D_i$ we have
$$\int_S D_iD_j=\begin{cases}
0&\textup{ if }i\neq j \textup{ or }i=j=0\\
-1& \textup{ if }i=j>0
\end{cases}$$
and 
$$\int_S D_i c_1=2+D_i^2=\begin{cases}
2&\textup{ if }i=0\\
1& \textup{ if }i>0
\end{cases}$$
We refer to \cite[Theorems 8.2.3, 10.4.4]{cox} for the properties of toric surfaces required for this. Hence 
\[\int_S \gamma_i\gamma_j=-\sum_{k=1}^m a_{ik}a_{jk}\]
and
\[\int_S \gamma_i c_1=2a_{i0}+\sum_{j=1}^m a_{ij}.\]

If we let $a=\{a_{i0}\}_{1\leq i\leq m}\in \QQ^{m}$ and $A=\{a_{ij}\}_{1\leq i, j\leq m}\in M_{m\times m}(\QQ)$ we want to show that the map
\begin{align*}M_{m\times m}(\QQ)\times \QQ^m&\to \textup{Sym}_m(\QQ)\times \QQ^m\\
(A, a)&\mapsto (-AA^t, 2a+A\mathbf{1})
\end{align*}
has a Zariski dense image. Here $\mathbf{1}=(1, \ldots, 1)^t$ and $\textup{Sym}_m(\QQ)$ denotes the set of $m\times m$ symmetric matrices. To show this, it's enough to show that the map
\begin{align*}M_{m\times m}(\RR)&\to \textup{Sym}_m(\RR)\\
A&\mapsto -AA^t\end{align*}
has Zariski dense image since $M_{m\times m}(\QQ)$ is dense inside $M_{m\times m}(\RR)$. But the image of the latter map is precisely the set of negative semi-definite matrices, which is open in the standard topology and hence Zariski dense.
\end{proof}

\subsection{Proof of theorem \ref{main}: $(p,p)$ insertions}
\label{subsec: mainhodge}

\noindent We'll begin now the proof of theorem \ref{main}
with the case of $(p,p)$ insertions. More precisely, we'll prove theorem \ref{main} when $D$ is in the algebra $\DD^S_0$ generated by

\[\{\ch_k(\gamma): \gamma\in H^{p,p}(S) \textup{ for some }p=0,1,2\}.\]

The ingredients for this step are the universality statement in lemma \ref{universalvirasorohodge}, the result for toric surfaces proven previously (theorem \ref{maintoric}) and the Zariski density of proposition \ref{zariskidensity}.

\begin{proposition}
\label{basecase}
Theorem \ref{main} holds when $D\in \DD^S_0$ is in the algebra generated by descendents of $(p,p)$ classes.  
\end{proposition}

\begin{proof}
By corollary \ref{disconnected} it's enough to prove the result when $S$ is connected, and in that case we may assume that $D$ has the form
\[D=\prod_{i=1}^s \ch_{k_i}(\1)\prod_{i=1}^t \ch_{\ell_i}(\p)\prod_{i=1}^m \ch_{m_i}(\gamma_i)\]
where $s,t,m, k_i, \ell_i, m_i\geq 0$ are integers, $\gamma_i\in H^{1,1}(S)$ and $\p\in H^4(S)$ is such that $\int_S \p=1$. By lemma \ref{universalvirasorohodge}, if we fix $k,s,t,m, k_i, \ell_i, m_i$ there is a polynomial in $\binom{m+1}{2}+m+2$ variable $F$ such that 
\[\langle \L_k D\rangle^S_n=F\left(\left\{\int_S\gamma_i\gamma_j\right\}_{1\leq i\leq j\leq m}, \left\{\int_S \gamma_i c_1\right\}_{1\leq i\leq m}, \int_S c_1^2, \int_S c_2\right).\]

Since the result holds for (disconnected) toric surfaces by corollary \ref{disconnectedtoric} and by proposition \ref{zariskidensity}, the polynomial $F$ vanishes in a Zariski dense set, and thus is identically 0. 
\end{proof}

\subsection{Proof of theorem \ref{main}: Non-$(p,p)$ insertions}
\label{subsec: mainnonhodge}

\noindent For the proof in the general case, we proceed by induction on the amount of non-$(p,p)$ insertions. To be more precise,  we consider a basis $\alpha_1, \ldots, \alpha_{h^{0,2}}$ of $H^{0,2}(S)$ and its dual basis $\beta_1,\ldots, \beta_{h^{0,2}}\in H^{2,0}(S)$, that is,
\[\int_S \alpha_i\beta_j=\delta_{ij}.\]
The algebra $\DD^S$ admits a filtration
\[\DD^S_0\subseteq \DD^S_1\subseteq \ldots \subseteq \DD^S_{h^{0,2}}=\DD^S\]
defined as follows: $\DD^S_l$ is the algebra generated by descendents of $(p,p)$ classes and descendents of $\alpha_1, \ldots, \alpha_l, \beta_1, \ldots, \beta_l$. In particular, $\DD^S_0$ agrees with the previous definition. We prove that theorem \ref{main} holds for every $D\in \DD^S_l$ by induction on $l$. The base case was proposition \ref{basecase} in the previous section. From now on we fix $l$ and assume that \ref{main} holds for $D_0\in \DD^S_{l-1}$. We want to show that for any $s,t, k_1, \ldots, k_s,\ell_1, \ldots, \ell_t$ and $D_0\in \DD^S_{l-1}$ we have
\[\left\langle \L_k\left(D\right)\right\rangle^S_n=0\]
where 
\begin{equation}\label{Dinduction}
D=\left(\prod_{i=1}^s\ch_{k_i}(\alpha_l)\prod_{i=1}^t\ch_{\ell_i}(\beta_l)\right)D_0
\end{equation}

Once again we may assume that $S$ is connected, and thus we can write
\[D_0=\prod_{i=1}^u \ch_{m_i}(1_S)\prod_{i=1}^v \ch_{n_i}(\gamma_i)\]
where $1_S\in H^0(S)$ is the fundamental class and $\gamma_i$ are classes either in $H^{p,p}(S)$ for $p>0$ or in $\{\alpha_1, \ldots, \alpha_{l-1}, \beta_1, \ldots, \beta_{l-1}\}$. In either case we have $\alpha_l\gamma_i=0=\beta_l\gamma_i$. 

The idea to deal with the non-$(p,p)$ classes $\alpha_l, \beta_l$ is to add more (toric) connected components and replace the classes $\alpha_l,\beta_l$ with classes in $H^0$ and $H^4$ of the new connected component. We define
\[E=E_1\sqcup \ldots \sqcup E_N\textup{ and }T=S\sqcup E\]
where $E_1, \ldots, E_N$ are $N$ copies of $\PP^2$ (or any other toric surface) and $N>s$. We let $1_i\in H^0(T)$ and $\p_i\in H^4(T)$ denote the fundamental class $[E_i]$ and the point class $\p_i$ of the connected component $E_i$. Similarly we consider $1_S\in H^0(T)$. Let 
\[\1'=\sum_{i=1}^N1_i\textup{ and }\1=\1'+1_S.\]
Note that $\1$ is the unit of $H^\ast(T)$. We denote
\[\tilde D_0=\prod_{i=1}^u \ch_{m_i}(\1)\prod_{i=1}^v \ch_{n_i}(\gamma_i)\in \DD^T_{l-1}\subseteq \DD^T.\]

We will now introduce two new classes in $H^0(T;\CC)$ and $H^4(T;\CC)$. Note that we can extend the definitions of descendents to allow classes in $H^\ast(S; \CC)$; we replace the algebra $\DD^S$ by $\DD^S\otimes \CC$, extend $\L_k$ linearly and everything we previously said (for example the universality statements in \ref{universal} and \ref{universalvirasorohodge}) still holds. We let
\[\alpha=\sum_{i=1}^N \omega^i 1_i\in H^0(T; \CC) \textup{ and }\beta=\frac{1}{N}\sum_{i=1}^N \omega^{-i} \p_i\in H^4(T; \CC)\]
where $\omega=e^{\frac{2\pi i}{N}}$ is a primitive $N$-th root of unity. 

\begin{claim}The sets of integrals of
\[\left(\underbrace{\alpha_l, \ldots, \alpha_l}_{s}, \underbrace{\beta_l, \ldots, \beta_l}_t, \underbrace{\1, \ldots, \1}_u,\gamma_1, \ldots, \gamma_v\right)\]
and 
\[\left(\underbrace{\alpha, \ldots, \alpha}_{s}, \underbrace{\beta, \ldots, \beta}_t, \underbrace{\1, \ldots, \1}_u,\gamma_1, \ldots, \gamma_v\right)\]
are the same.
\end{claim}

\begin{proof}We have by construction
\begin{equation}\label{integral1}\int_S \alpha \beta=\int_T\left(\frac{1}{N}\sum_{i=1}^N \p_i\right)=1
\end{equation}
and, for $j=0, 2, \ldots, s$,
\begin{equation}\label{integral2}\int_T \alpha^j\beta=\frac{1}{N}\sum_{i=0}^N \omega^{(j-1)i}=0
\end{equation}
since $s<N$. Similarly 
\begin{equation}\label{integral3}\int_S \alpha^j c_1(T)^2=0 \textup{ and }\int_S \alpha^j c_2(T)=0
\end{equation}
for $j=1, \ldots, s$. Equations \eqref{integral1}, \eqref{integral2} and \eqref{integral3} also hold replacing $\alpha, \beta$ by $\alpha_l, \beta_l$, and moreover we have 
\[\alpha_l\gamma_i=\beta_l\gamma_i=\alpha\gamma_i=\beta\gamma_i=0 \textup{ for }i=1, \ldots, v.\]
These facts prove the claim.
\end{proof}

By the claim and by proposition \ref{universalvirasorohodge} it follows that
\begin{align}\label{vanishing11}\nonumber
&\left\langle\L_k\left(\left(\prod_{i=1}^s\ch_{k_i}(\alpha_l)\prod_{i=1}^t\ch_{\ell_i}(\beta_l)\right) \tilde{D}_0 \right) \right \rangle^T_n\\
&\textup{ }\quad\quad \quad=\left\langle\L_k\left(\left(\prod_{i=1}^s\ch_{k_i}(\alpha)\prod_{i=1}^t\ch_{\ell_i}(\beta)\right)\tilde{D}_0 \right) \right \rangle^T_n
=0.\end{align}
The vanishing holds by the induction hypothesis since 
\[\left(\prod_{i=1}^s\ch_{k_i}(\alpha)\prod_{i=1}^t\ch_{\ell_i}(\beta)\right)\tilde{D}_0\in \DD^T_{l-1}.\]

Equation \eqref{vanishing11} is almost what we wanted except that we replaced the appearances of $1_S$ in $D_0$ by $\1=1_S+\1'$ in $\tilde D_0$. If there are no such appearances, i.e. $u=0$, then $D_0=\tilde D_0$ and by~\eqref{virasorodisconnected}
\begin{equation}\label{TtoS}0=\left\langle \L_k^T(D)\right\rangle^T_n=\left\langle \L_k^S(D)\right\rangle^S_n+\left\langle D\right\rangle^S_{n-k/2}\left\langle \L_k^E(1_{\DD^E})\right\rangle^E_{k/2}.
\end{equation}
Since we know already that $\left\langle \L_k^E(1_{\DD^E})\right\rangle^E_{k/2}=0$ vanishes, it follows that $\left\langle \L_k^S(D)\right\rangle^S_n=0$ also vanishes. Here the $1_{\DD^E}$ denotes the unit in the algebra $\DD^E$.

To finish the proof we now argue by induction on $u$. We abbreviate
\[B=\prod_{i=1}^s\ch_{k_i}(\alpha_l)\prod_{i=1}^t\ch_{\ell_i}(\beta_l)\prod_{i=1}^v \ch_{n_i}(\gamma_i).\]
Now we have
\begin{align*}0=\left\langle \L_k^T\left(B\prod_{i=1}^u\ch_{m_i}(\1)\right)\right \rangle ^T_n&=\left\langle \L_k^T\left(D\right)\right \rangle ^T_n\\
&+\sum_{I\subsetneq [u]}\left\langle \L_k^T\left(\left(\prod_{i\in I}\ch_{m_i}(1_S)\prod_{i\in [u]\setminus I}\ch_{m_i}(\1')\right)B\right)\right \rangle ^T_n.
\end{align*}
For $I\subsetneq [u]$ we can write by equation \ref{virasorodisconnected}
\begin{align*}
&\left\langle \L_k^T\left(\left(\prod_{i\in I}\ch_{m_i}(1_S)\prod_{i\in [u]\setminus I}\ch_{m_i}(\1')\right)B\right)\right \rangle ^T_n\\
&\quad \quad \quad=\sum_{n_1+n_2=n}\left\langle \L_k^S\left(\left(\prod_{i\in I}\ch_{m_i}(1_S)\right)B\right)\right \rangle^S_{n_1}\left\langle \prod_{i\in [u]\setminus I}\ch_{m_i}(\1') \right \rangle ^E_{n_2}\\
&\quad \quad \quad+
\sum_{n_1+n_2=n}\left\langle \left(\prod_{i\in I}\ch_{m_i}(1_S)\right)B\right \rangle^S_{n_1}\left\langle\L_k^E\left( \prod_{i\in [u]\setminus I}\ch_{m_i}(\1') \right)\right \rangle ^E_{n_2}
\end{align*}
and this expression must vanish: the second line vanishes by the induction hypothesis on $u$ since $|I|<u$ and the third line vanishes since theorem \ref{main} holds for $E$. So we conclude that 
\[\langle\L_k^T(D)\rangle^T_n=0\] and again using \eqref{TtoS} we find
\[\langle\L_k^S(D)\rangle^S_n=0.\]

\section{The cubic 3-fold}
\label{sec: cubic}

\noindent Let $X\subseteq \PP^4$ be a smooth cubic hypersurface and let $F\in H^0(\PP^4, \O(3))$ be the degree 3 polynomial defining $X$.

Moreover let $F(X)$ be the Fano variety of lines in $X$, i.e.,
\[F(X)=\{\ell\in \GG(1, 4): \ell \subseteq X\}.\]
Here $\GG(1,4)=G(2,5)$ is the Grassmanian of lines on $\PP^4$ or, equivalently, the Grassmanian of 2-subspaces of $\CC^5$.

\subsection{Basic facts about $X$}

\noindent Let $j: X\hookrightarrow \PP^4$ be the inclusion. We will denote by $H\in H^2(\PP^4)$ the hyperplane class and, when confusion doesn't arise, we'll also denote by $H$ the pullback $j^\ast H\in H^2(X)$.  By the Lefschetz hyperplane theorem $j^\ast$ induces an isomorphism $H^{k}(X)\cong H^k(\PP^4)$ for $k<3$ and $j_\ast$ induces an isomorphism $H^{k}(X)\cong H^{k+2}(\PP^4)$ for $k>3$; moreover
\[j_\ast j^\ast H^j=[X]H^j=3H^{j+1}\in H^\ast(\PP^4).\]
Thus $H^\ast(X)$ is generated outside degree 3 by $1, H, \frac{1}{3}H^2, \frac{1}{3}H^3$. 

The Chern class of $X$ is computed via the normal sequence to get
\[c(X)=\frac{j^\ast c(\PP^4)}{j^\ast c(\O_{\PP^4}(3))}=j^\ast\frac{(1+H)^5}{1+3H}=1+2H+4H^2-2H^3.\]
In particular, $\chi(X)=\int_X (-2H^3)=-6$ so it follows that
\[b_3(X)=10.\]

More generally, we can compute $\chi_{-y}$ and get the Hodge numbers $h^{3,0}=h^{0,3}=0$ and $h^{2,1}=h^{1,2}=5$ (see \cite[Theorem 1.11]{cubic} and the table afterwards). 

The Gromov-Witten theory of $X$ is reasonably understood in genus 0, but very hard to compute in higher genus. We refer to \cite{quantumcohprimitivehu} for a reconstruction theorem of genus 0 Gromov-Witten invariants of $X$ and a discussion about some higher genus invariants. The Gromov-Witten Virasoro constraints are not known for $X$.

\subsection{Basic facts about $F(X)$}

\noindent Denote by $\S$ the tautological rank 2 bundle over the Grassmannian $\GG(1,4)$ and by $\Q=\O_{\GG}^{\oplus 5}/\S$ the quotient rank 3 bundle. The Fano variety $F(X)$ is a smooth closed 2-dimensional subvariety of $\GG(1, 4)$ (see \cite[Corollary 1.14]{cubic}). The Fano variety can be described as the zero-set of a section $s_F$, canonically determined by $F$, of the rank 4 bundle $\Sym^3(\S^\ast)$ over $\GG(1, 4)$. In particular, we can compute the class $[F(X)]\in H_4(\GG(1,4))\cong H^8(\GG(1, 4))$ in terms of the Chern classes $c_1=c_1(\S)$ and $c_2=c_2(\S)$ of the tautological bundle $\S$:
\[[F(X)]=c_4(\Sym^3(\S^\ast))=18c_1^2c_2+9c_2^2\in H^8(\GG(1, 4)).\]

We also denote by $c_1, c_2$ the pullbacks of $c_1, c_2$ to $F(X)$ via the inclusion $F(X)\hookrightarrow \GG(1,4)$; note that $c_1=-g$ where $g$ is the Plücker polarization. It will later be useful to have the following integrals:

\begin{equation}
\label{integralsfano}
\int_{F(X)} c_1^2=45\textup{ and }\int_{F(X)} c_2=27.
\end{equation}
These are computed using the expression of $[F(X)]$ and the relations
\[0=c_4(\Q)=c_2^2-3c_1^2c_2+c_1^4 \textup{ and }0=c_5(\Q)=-3c_1c_2^2+4c_1^3c_2-c_1^5\]
in $H^\ast(\GG(1, 4))$ between the generators $c_1, c_2$. From those we also have 
\[2c_2^3=2c_1^2c_2^2=c_1^4c_2.\]
The computation is then finished with $\int_{\GG(1,4)}c_2^3=1$ (see \cite[Corollary 4.2]{eisenbudharris}).

A description of the Hodge structure of $F(X)$ is given in \cite{cubic} and we will quickly explain it. We introduce the universal line $\L=\PP(\S_{|F(X)})$; set theoretically $\L$ is described as
\[\L=\{(x, \ell)\in X\times F(X): x\in \ell\}.\]
Let $\pi^\L_X: \L\to X$ and $\pi^\L_F: \L\to F(X)$ be the obvious projections. We let 
\[\varphi=(\pi_F^\L)_\ast(\pi_X^\L)^\ast: H^3(X)\to H^1(F(X)).\]
It's proven in \cite[Proposition 4.2]{cubic} that $\varphi$ is an isomorphism. In particular, we get the Hodge numbers $h^{1,0}(F(X))=5=h^{0,1}(F(X))$. By \cite[Lemma 2.3]{cubic} the product on cohomology induces an isomorphism $\wedge^2 H^1(F(X))\cong H^2(F(X))$, thus $h^{2,0}(F(X))=h^{0,2}(F(X))=10$ and $h^{1,1}(F(X))=25$. Finally, \cite[Proposition 4.2]{cubic} also gives the identity
\begin{equation}
\int_{F(X)}\varphi(\alpha)\varphi(\beta)c_1=6\int_X \alpha \beta \textup{ for all }\alpha, \beta\in H^3(X).
\end{equation}

\subsection{Virasoro conjecture in the line class of the cubic 3-fold}

\noindent In the next two sections we'll explain how to compute the full theory of stable pairs with descendents for the line class of the cubic 3-fold. We state here the list of all the relevant partition functions:

\begin{theorem}
\label{fullPTcubic}
Let $X$ be the cubic 3-fold and $\beta\in H_2(X; \beta)$ be the line class. Writing $Z_{\textup{PT}}(D)$ for $Z_{\textup{PT}}^X(q|D)_\beta$ we have:
\begin{align}
&Z_{\textup{PT}}(\ch_4(1) \ch_4(1))=\frac{5(q-44q^2+126q^3-44q^4+q^5)}{4(1+q)^4}\\
&Z_{\textup{PT}}(\ch_4(1) \ch_3(H))=\frac{15(q-5q^2+5q^3-q^4)}{4(1+q)^3}\\
&Z_{\textup{PT}}(\ch_4(1) \ch_2(H^2))=\frac{15(-q+4q^2-q^3)}{2(1+q)^2}\\
&Z_{\textup{PT}}(\ch_3(H) \ch_3(H))=\frac{45q}{4}\\
&Z_{\textup{PT}}(\ch_3(H) \ch_2(H^2))=\frac{45(-q+q^2)}{2(1+q)}\\
&Z_{\textup{PT}}(\ch_2(H^2) \ch_2(H^2))=45q\\
&Z_{\textup{PT}}(\ch_5(1))=\frac{15(q-5q^2+5q^3-q^4)}{4(1+q)^3}\\
&Z_{\textup{PT}}(\ch_4(H))=\frac{21q}{4}\\
&Z_{\textup{PT}}(\ch_3(H^2))=\frac{45(-q+q^2)}{2(1+q)}\\
&Z_{\textup{PT}}(\ch_2(H^3))=18q\\
&Z_{\textup{PT}}(\ch_2(\gamma)\ch_3(\gamma'))=\frac{3(q-q^2)}{1+q}\int_X \gamma \gamma'\\
&Z_{\textup{PT}}(\ch_2(\gamma)\ch_2(\gamma')\ch_4(1))=\frac{q-4q^2+q^3}{(1+q)^2}\int_X \gamma \gamma'\\
&Z_{\textup{PT}}(\ch_2(\gamma)\ch_2(\gamma')\ch_3(H))=\frac{3(q-q^2)}{1+q}\int_X \gamma \gamma'\\
&Z_{\textup{PT}}(\ch_2(\gamma)\ch_2(\gamma')\ch_2(H^2))=-6q\int_X \gamma \gamma'\\
&Z_{\textup{PT}}(\ch_2(\gamma_1)\ch_2(\gamma_2)\ch_2(\gamma_3)\ch_2(\gamma_4))=q\left(
\left(\int_X \gamma_1\gamma_2\right)\left(\int_X \gamma_3\gamma_4\right)\right. \nonumber \\
&\left.\quad+\left(\int_X \gamma_1\gamma_4\right)\left(\int_X \gamma_2\gamma_3\right)
+\left(\int_X \gamma_1\gamma_3\right)\left(\int_X \gamma_4\gamma_2\right)\right)
\end{align}
for $\gamma, \gamma', \gamma_i\in H^3(X)$.
In particular, conjecture \ref{conj: rationality} (rationality and functional equation) holds in this case.
\end{theorem}

This calculation will be explained, modulo the computational steps, in the next two sections; section \ref{cubic3foldn1} will compute the coefficient of $q^1$ in these partition functions and, using that, we'll compute the full partition function in section \ref{cubic3foldngeneral}.

The explicit computation allows us to verify the Virasoro constraints in this case, proving theorem \ref{maincubic}. Indeed, it's enough to check a finite amount of relations since, according to the next proposition, we can restrict ourselves to products of descendents with positive cohomological degree. The next proposition holds in general for any $X$, $\beta$ and not only for the cubic 3-fold with the line class.

\begin{proposition}
\label{reductionvirasoro}
Suppose that conjecture \ref{conj: virasoro} holds for some $D\in \DDpt$, that is,
\[\langle \L_k D\rangle_{n, \beta}^{X, \textup{PT}} =0.\]
Then conjecture \ref{conj: virasoro} also holds for 
\[\ch_0(\gamma)D, \ch_1(\gamma)D, \ch_2(1)D, \ch_2(\delta)D, \ch_3(1)D\]
for any $\gamma\in H^\ast(X)$, $\delta\in H^{1,1}(X)$.
\end{proposition}
\begin{proof}
All of these are fairly easy verifications using the expressions for $\ch_0, \ch_1$ and the string, divisor and dilaton equations from proposition \ref{descendentequations}.
\begin{enumerate}
\item We have
\[\L_k (\ch_0(\gamma)D)=(R_k\ch_0(\gamma))D+\ch_0(\gamma)\L_k(D);\]
by our assumption on $D$ it follows that $\langle \ch_0(\gamma)\L_k(D)\rangle=0$. Moreover if $\gamma\in H^{p,q}(X)$
\[R_k\ch_0(\gamma)=\left(\prod_{n=0}^k (p+n-3)\right) \ch_k(\gamma);\]
if $p+k-3\geq 0$ then the product vanishes; otherwise $\langle\ch_k(\gamma)D \rangle^{X, \textup{PT}}_{n,\beta}=0$ because $p+k-3$ is the first Hodge degree of $\ch_k(\gamma)\in H^\ast(P_n(X, \beta))$. 
\item We have 
\begin{align*}\L_k (\ch_1(\gamma)D)=&(R_k\ch_1(\gamma))D+\ch_1(\gamma)\L_k(D)\\&+(k+1)!\sum_{p_i^L=0}\left(R_{-1}[\gamma_i^L]\ch_1(\gamma)\right)\ch_{k+1}(\gamma_i^R)D.
\end{align*}
The bracket of the middle term vanishes by definition and 
\[R_k\ch_1(\gamma)=\left(\prod_{n=0}^k (p+n-2)\right)\ch_{1+k}(\gamma).\]
If $p< 3$ the same argument as before shows that $\langle R_k\ch_1(\gamma)\rangle^{X, \textup{PT}}_{n,\beta}=0$, and also the last term vanishes since 
\[R_{-1}[\gamma_i^L]\ch_1(\gamma)=\ch_0(\gamma\gamma_i^L)=-\int_X \gamma\gamma_i^L=0.\]
If $p=3$ then the first term is $(k+1)!\ch_{k+1}(\gamma)D$ and the last term (after collapsing $\ch_0, \ch_1$) is
\begin{align*}-(k+1)!\sum_{p_i^L=0}\left(\int_{X}\gamma_i^L\gamma\right)\ch_{k+1}(\gamma_i^R)D=-(k+1)!\ch_{k+1}(\gamma)D.
\end{align*}
\item We have
\begin{align*}\L_k (\ch_2(1)D)=&(R_k\ch_2(1))D+\ch_2(1)\L_k(D)\\
&+(k+1)!\sum_{p_i^L=0}\ch_1(\gamma_i^L)\ch_{k+1}(\gamma_i^R)D.
\end{align*}
Applying the bracket to the last two terms gives 0 immediately since $\ch_2(1)=\ch_1(\gamma_i^L)=0$. Repeating the previous argument $R_k \ch_2(1)=0$.
\item We have 
\begin{align*}\L_k (\ch_2(\delta)D)=&(R_k\ch_2(\delta))D+\ch_2(\delta)\L_k(D)\\
&+(k+1)!\sum_{p_i^L=0}\ch_1(\delta\gamma_i^L)\ch_{k+1}(\gamma_i^R)D.
\end{align*}
The bracket of the first and the last terms vanish once again. By the divisor equation and by hypothesis:
\[\langle \L_k(\ch_2(\delta)D)\rangle_{n, \beta}^{X, \textup{PT}}=\langle \ch_2(\delta)\L_k D\rangle_{n, \beta}^{X, \textup{PT}}=\left(\int_{\beta}\delta\right) \langle \L_k D\rangle_{n, \beta}^{X, \textup{PT}}=0.\]
\item We have
\begin{align*}\L_k (\ch_3(1)D)=&(R_k\ch_3(1))D+\ch_3(1)\L_k(D)\\
&+(k+1)!\sum_{p_i^L=0}\ch_2(\gamma_i^L)\ch_{k+1}(\gamma_i^R)D.
\end{align*}
Once again the first and last term vanish; for the last term we use the $\gamma_i^L$-string equation, proposition \ref{descendentequations} $i)$. The bracket of the middle term vanishes by the assumption that $D$ satisfies the Virasoro constraint and by the dilaton equation.
\end{enumerate}
\end{proof}

Since the virtual dimension of $P_n(X, \beta)$ is $d_\beta=2$, we only have to check $\langle \L_k(D)\rangle^{X, \textup{PT}}_{n,\beta}=0$ for $2k+|D|=2d_\beta=4$. Moreover proposition \ref{reductionvirasoro} reduces us to the cases where $D$ is a product of descendents $\ch_i(\gamma)$ with $i\geq 2$ and $2i+|\gamma|-6>0$. We know already the result holds for $k=-1$ and $k=0$ so this leaves us with the following cases:
\begin{enumerate}
\item $k=2$ and $D=1$;
\item $k=1$ and $D=\ch_j(\gamma)$ for $(j, \gamma)\in \{(2, H^2), (3, H), (4, 1)\}$;
\item $k=1$ and $D=\ch_2(\gamma)\ch_2(\gamma')$ for $\gamma\in H^{1,2}(X), \gamma'\in H^{2,1}(X)$. 
\end{enumerate}

We'll check these cases by hand. We have $c_1=2H$,
\[\Delta_\ast c_1=\frac{2}{3}(H\otimes H^3+H^2\otimes H^2+H^3\otimes H)\]
 and $c_1c_2=8H^3=24\p$. After collapsing $\ch_0, \ch_1$ we have the following expressions for $\L_1, \L_2$:
\begin{align*}&\L_1=R_1-2\ch_3(H)+\frac{2}{3}\ch_2(H^3)R_{-1}\\
&\L_2=R_2-4\ch_4(H)+\frac{4}{3}\ch_2(H)\ch_2(H^3)-\frac{1}{3}\ch_2(H^2)\ch_2(H^2)-\frac{4}{3}\ch_2(H^3)+2\ch_2(H^3).
\end{align*}

Using that $\ch_2(H)=1$, by the divisor equation, case $(1)$ turns out to be equivalent to the identity
\[-4Z_{\textup{PT}}(\ch_4(H))-\frac{1}{3}Z_{\textup{PT}}(\ch_2(H^2)\ch_2(H^2))+2Z_{\textup{PT}}(\ch_2(H^3))=0\]
which is equivalent to
\[-4\frac{21q}{4}-\frac{45q}{3}+2\times 18q=0.\]

Case $(2)$ is equivalent to 
\[Z_{\textup{PT}}(\ch_{j+1}(\gamma))-Z_{\textup{PT}}(\ch_3(H)\ch_j(\gamma))+\frac{1}{3}Z_{\textup{PT}}(\ch_2(H^3)\ch_{j-1}(\gamma))=0.\]
These relations are checked for $(j, \gamma)=(1, H^3), (2, H^2), (3, H), (4, 1)$ using theorem \ref{fullPTcubic}.

Finally, case $(3)$ turns into
\[Z_{\textup{PT}}(\ch_2(\gamma)\ch_3(\gamma'))-Z_{\textup{PT}}(\ch_2(\gamma)\ch_2(\gamma')\ch_3(H))=0\] 
which also holds by the computations in theorem \ref{fullPTcubic}.

\section{Computation in $P_1(X, \beta)$}
\label{cubic3foldn1}

We are interested in computing the stable pairs theory $P_{n+1}(X, \beta)$ for $X$ when the curve class $\beta$ is the class of a line in $X$, that is, $\beta=\frac{1}{3}H^2$. The virtual dimension of $P_{n+1}(X, \beta)$ is given by
\[\int_\beta c_1(TX)=\int_X \frac{2}{3}H^3=2.\]

We can describe explicitly what is $P_{n+1}(X, \beta)$. Since the support of a stable pair in $P_{n+1}(X, \beta)$ is necessarily a line $L\in F(X)$, and in particular is Gorenstein, by the results in \cite[Appendix B]{stablepairsbps} it follows that stable pairs supported in $L$ are in correspondence with 0-dimensional subschemes of $L$ or, equivalently, effective divisors on $L$. 

Given such a divisor $D$, the Euler characteristic of the associated stable pair is $|D|+1-g(L)=|D|+1$, by Riemann-Roch for curves. Thus, $P_{n+1}(X, \beta)$ is a bundle over $F(X)$ with fiber $L^{[n]}$ over $L\in F(X)$. Here $L^{[n]}$ means the $n$-fold symmetric product of $L$, which parametrizes degree $n$ effective divisors on $L$. 

It follows from this description that $P_{n+1}(X, \beta)$ is a smooth projective variety of dimension $2+n$. When $n=0$ then $P_1(X, \beta)=F(X)$ and its actual dimension matches its virtual dimension 2. So in this case $[P_1(X, \beta)]^{\textup{vir}}$ is the fundamental class of $P_1(X, \beta)=F(X)$.

\subsection{Computing descendents, $n=0$}
\label{descendentsn0}
We will now compute all the descendents in $P_1(X, \beta)$. We introduce the following maps which we'll use during the computations:

\begin{center}
\begin{tikzcd}
\,&X& \PP^4\times F(X)\arrow[d, hookrightarrow, "j_2"]\\
\mathcal L \arrow[r,hookrightarrow, "\iota"]\arrow[ru, "\pi_X^\L"]\arrow[rd, "\pi_F^\L"']&
X\times F(X)\arrow[r,hookrightarrow, "j"]\arrow[u, "\pi_X"']\arrow[d, "\pi_F"]\arrow[ru,"j_1"]&
\PP^4\times \GG(1,4)\\
\,&F(X)&\,
\end{tikzcd}
\end{center}

The first observation is that the universal stable pair is $\FF=\iota_\ast \O_\L$. Indeed when we restrict $\O_{X\times F(X)}\to \iota_\ast \O_\L$ to $X\times \{L\}\subseteq X\times F(X)$ we get the corresponding stable pair $\O_X\to i_\ast \O_L$. Hence we can use Grothendieck-Riemann-Roch to compute the Chern character of $\FF$.

\[\ch(\FF)=\iota_\ast\left(\ch(\O_\L)\textup{td}(-\mathcal N_{\L/X\times F(X)})\right)=\iota_\ast \textup{td}(-\mathcal N_{\L/X\times F(X)}).\]
We relate this normal bundle to the normal bundles of $\L$ and $X\times F(X)$ inside $\PP^4\times \GG(1, 4)$ using the exact sequence
\[0\rightarrow \N_{\L/X\times F(X)}\rightarrow \N_{\L/\PP^4\times \GG(1,4)} \rightarrow \iota^\ast\N_{X\times F(X)/\PP^4\times \GG(1,4)} \rightarrow 0.\]

Moreover, the normal bundles inside $\PP^4\times \GG(1, 4)$ can be identified by writing $\L$ and $X\times F(X)$ as zero locus of (dimensionally transverse) sections of bundles. 

Clearly $X\times F(X)$ is the zero locus of the section $F\oplus s_F$ of the rank-4 bundle $\O_{\PP^4}(3)\oplus \Sym^3(\S^\ast)$. Regarding $\L$, we can write $\L$ as a dimensionally transverse intersection $\widetilde \L\cap (\PP^4\times F(X))$ where $\widetilde \L=\{(x, L)\in \PP^4\times \GG(1,4): x\in L\}$. Now $\widetilde \L$ can be described as the zero locus of a section of $\O_{\PP^4}(1)\boxtimes \Q$: given $(x, L)\in \PP^4\times \GG(1,4)$, we have a homomorphism
\[\O_{\PP^4}(-1)_x\to \CC^5\to \Q_L\]
determining a section of $\mathcal Hom(\O_{\PP^4}(-1), \Q)\cong \O_{\PP^4}(1)\boxtimes \Q$ whose zero locus is $\widetilde \L$. Thus we compute 
\begin{align}\ch(\FF)&=\iota_\ast \textup{td}(-\mathcal N_{\L/X\times F(X)})=\iota_\ast \frac{\iota^\ast j^\ast\textup{td}\left(\O(3)\oplus \Sym^3(\S^\ast)\right)}{\iota^\ast j^\ast\textup{td}\left(\O(1)\boxtimes \Q\oplus \Sym^3(S^\ast)\right)}\\
&=[\L]j^\ast \frac{\textup{td}\left(\O(3)\right)}{\textup{td}\left(\O(1)\boxtimes \Q\right)}.
\end{align}
where we used the push-pull formula for $\iota_\ast\iota^\ast \alpha=(\iota_\ast 1)\alpha=[\L]\alpha$ and wrote $[\L]$ for the class in $H_6(X\times F(X))\cong H^4(X\times F(X))$. Now
\[\textup{td}(\O(3))=\frac{3H}{1-e^{-3H}}\]
and 
$\textup{td}\left(\O(1)\boxtimes \Q\right)$
can be computed formally with the splitting principle: we get



\begin{align*}j^\ast \frac{\textup{td}\left(\O(3)\right)}{\textup{td}\left(\O(1)\boxtimes \Q\right)}=&1 - \frac{1}{2}c_1 + \frac{1}{6}c_1^2 - \frac{1}{12}c_2 + \frac{1}{12}H c_1 - \frac{1}{24}H c_1^2  + \frac{1}{4}H^2-\frac{1}{8}H^2c_1 \\
&+ \frac{31}{720} H^2c_1^2- \frac{1}{60}H^2c_2  + \frac{7 }{360}H^3c_1 - \frac{ 7 }{720} H^3c_1^2
\end{align*}
 
Computing $[\L]\in H^4(X\times F(X))$ is not straightforward since $\L$ is the zero locus of $j^\ast \O(1)\boxtimes \Q$ but $\L$ has codimension 2 inside $X\times F(X)$ while $j^\ast \O(1)\boxtimes \Q$ has rank 3. However, we can compute the pushforward of $[\L]$ to $\PP^4\times F(X)$ as 
\[(j_1)_\ast[\L]=j_2^\ast c_3(\O(1)\boxtimes \Q)\in H^{6}(\PP^4\times F(X)).\] The pushforward $(j_1)_\ast$ kills the component of $H^3(X)\otimes H^1(F(X))$ in the Künneth decomposition of $H^4(X\times F(X))$ but we can recover the rest, finding
\[[\L]=\frac{1}{3}H^2 - \frac{1}{3}H c_1 + \frac{1}{3}(c_1^2 - c_2)+A\]
where $A\in H^3(X)\otimes H^1(F(X))$.

This is enough to compute all the even descendents, which we now list:

\begin{align}
&\ch_3(1)=0,\quad
\ch_2(H)=1\\
&\ch_4(1)=\frac{1}{6}c_1,\quad
\ch_3(H)=\frac{1}{2}c_1,\quad
\ch_2(H^2)=-c_1\\
&\ch_5(1)=\frac{1}{12}c_1^2,\quad
\ch_4(H)=-\frac{1}{12}c_1^2+\frac{1}{3}c_2\\
&\ch_3(H^2)=-\frac{1}{2}c_1^2,\quad
\ch_2(H^3)=c_1^2-c_2
\end{align}

Moreover the push-pull formula gives
\[\ch_2(\gamma)=(\pi_F)_\ast\iota_\ast\iota^\ast\pi_X^\ast \gamma=(\pi_F^\L)_\ast(\pi_X^\L)^\ast\gamma=\varphi(\gamma).\]
Finally, for $\gamma\in H^3(X)$
\[\ch_3(\gamma)=(\pi_F)_\ast\left(A\frac 12 c_1\pi_X^\ast\gamma\right)=\frac{1}{2}c_1\varphi(\gamma).\]

\subsection{Computing the invariants}
All the invariants that only have descendents of even classes are straightforward to compute using the integrals \eqref{integralsfano} of $c_1^2$ and $c_2$ in $F(X)$.

\begin{align}
&\langle \ch_4(1) \ch_4(1) \rangle_1=\frac{5}{4}, \quad
\langle \ch_4(1) \ch_3(H) \rangle_1=\frac{15}{4}\\
&\langle \ch_4(1) \ch_2(H^2) \rangle_1=-\frac{15}{2}, \quad
\langle \ch_3(H) \ch_3(H) \rangle_1=\frac{45}{4}\\
&\langle \ch_3(H) \ch_2(H^2) \rangle_1=-\frac{45}{2}, \quad
\langle \ch_2(H^2) \ch_2(H^2) \rangle_1=45\\
&\langle \ch_5(1)\rangle_1=\frac{15}{4}, \quad
\langle \ch_4(H)\rangle_1=\frac{21}{4}, \quad
\langle \ch_3(H^2)\rangle_1=-\frac{45}{2}, \quad
\langle \ch_2(H^3)\rangle_1=18
\end{align}

The invariants with two odd descendents are computed using the identity 
\[\int_{F(X)}\varphi(\gamma)\varphi( \gamma')c_1=6\int_X \gamma \gamma'.\] Note that this is enough to compute everything because the descendents of degree 2 are all proportional to $c_1$ and $\ch_3(\gamma)=\frac{1}{2}c_1\varphi(\gamma)$. Let $\gamma, \gamma'\in H^3(X)$. 

\begin{align}
&\langle \ch_2(\gamma)\ch_2(\gamma')\ch_4(1)\rangle_1=\int_X \gamma\gamma',\quad
\langle \ch_2(\gamma)\ch_2(\gamma')\ch_3(H)\rangle_1=3\int_X \gamma\gamma'\\
&\langle \ch_2(\gamma)\ch_2(\gamma')\ch_2(H^2)\rangle_1=-6\int_X \gamma\gamma',\quad
\langle \ch_2(\gamma)\ch_3(\gamma')\rangle_1=3\int_X \gamma\gamma'
\end{align}

Finally, we're missing the case of 4 odd descendents. The required integral is computed in \cite[Theorem 6.7, iv)]{quantumcohprimitivehu}:

\begin{align}\langle \ch_2(\gamma_1)&\ch_2(\gamma_2)\ch_2(\gamma_3)\ch_2(\gamma_4)\rangle=\int_{F(X)}\varphi(\gamma_1)\varphi(\gamma_2)\varphi(\gamma_3)\varphi(\gamma_4)\nonumber\\
&=\left(\int_X \gamma_1\gamma_2\right)\left(\int_X \gamma_3\gamma_4\right)+\left(\int_X \gamma_1\gamma_4\right)\left(\int_X \gamma_2\gamma_3\right)\nonumber \\
&\quad+\left(\int_X \gamma_1\gamma_3\right)\left(\int_X \gamma_4\gamma_2\right)
\end{align}

\section{Computation in $P_{n+1}(X, \beta)$, $n>0$}
\label{cubic3foldngeneral}

We now study descendents in the moduli space $P_{n+1}(X, \beta)$ for $n>0$.  

We introduced previously $\L$ as the universal line over $F(X)$. Let also $\L^{[n]}$ denote the $n$-fold symmetric product of $\L$ over $F(X)$, that is, 
\[\L^{[n]}=\left(\L\times_{F(X)}\ldots \times_{F(X)} \L\right)\big /\Sigma_n\]
where $\Sigma_n$ is the permutation group acting by permuting the factors of the product. Then $P_{n+1}(X, \beta)=\L^{[n]}$. Recall that\footnote{For us, a projective bundle $\PP E$ parametrizes 1-dimensional subspaces of $E$ (and not $1$-dimensional quotients).} $\L=\PP_{F(X)}(\S)$ where we still denote by $\S$ the restriction of the tautological bundle $\S$ on the Grassmanian $\GG(1,4)$ to $F(X)$. Hence
\[\L^{[n]}=\PP_{F(X)}(\Sym^n \S).\]
As a set:
\[\L^{[n]}=\{(L, D): L\in F(X), D\in \textup{Div}_{\textup{eff}}(L), |D|=n\}.\]

As a projective bundle, $\L^{[n]}$ carries a tautological line bundle $\O_{\L^{[n]}}(-1)$ (whose fiber over $L$ is identified with the line inside $\Sym^n \S$ corresponding to $L$). We denote by $\zeta_n$ the Chern class
\[\zeta_n=c_1\left(\O_{\L^{[n]}}(1)\right)\in H^2\left(\L^{[n]}\right).\]

By the projective bundle theorem the cohomology of $\L^{[n]}$ is
\[H^\ast(\L^{[n]})=H^\ast(F(X))[\zeta_n]/\left(\zeta_n^{n+1}+\zeta_{n}^n c_1(\Sym^n\S)+\ldots+c_{n+1}(\Sym^n\S)\right).\]

\subsection{The universal divisor}

There is a universal (effective) divisor $\D=\D_n$ in the fiber product $\L\times_{F(X)}\L^{[n]}$ such that its restriction to a fiber $\L\times_{F(X)} \{(L, D)\}\cong L$ is $D$. We can identify the class of $\D$ in $H^2\left(\L\times_{F(X)}\L^{[n]}; \ZZ\right)$. We will use, now and for the rest of the section, the maps $p, q, \pi_1, \pi_n$ which are the obvious projections in the pullback diagram:

\begin{center}
\begin{tikzcd}
\D\arrow[r, hookrightarrow]&\L\times_{F(X)}\L^{[n]}\arrow["q",r]\arrow["p", d]&
\L \arrow["\pi_1", d]\\
&
\L^{[n]}\arrow["\pi_n", r]&
F(X)
\end{tikzcd}
\end{center}

We still denote by $\zeta_n, \zeta_1, c_1$ the pull-backs of the original classes to $\L\times_{F(X)}\L^{[n]}$ via $p$, $q$ and $\pi_X^\L \circ q$, respectively.

\begin{proposition}
\label{universaldivisor}
We have
\[[\D]=\zeta_n+n\zeta_1+nc_1\]
in $H^2\left(\L\times_{F(X)}\L^{[n]}; \ZZ\right)$.
\end{proposition}
\begin{proof}
The result follows by identifying $\D$ as the vanishing locus of a (canonical) section $s$ of the line bundle
\[\mathcal{H}om\left(q^\ast \O_{\L}(-1)^{\otimes n}\otimes p^\ast \O_{\L^{[n]}}(-1), (\Lambda^2 \S)^{\otimes n}\right).\] 
Indeed this section can be described as follows: let $(L, x, D)\in \L\times_{F(X)} \L^{[n]}$ and consider the embeddings $\O_{\L}(-1)\hookrightarrow \S$ and $\O_{\L^{[n]}}(-1)\hookrightarrow \textup{Sym}^n \S$. Let $a_1\otimes \ldots\otimes a_n$ be in the fiber of $\O_\L(-1)^{\otimes n}$ over $(L, x)$ and $b_1\ldots b_n$ be in the fiber of $\O_{\L^{[n]}}(-1)$ over $(L, D)$, with $a_i, b_i\in \S_L$. Then the value of the section at $(L, x, D)$ is the morphism
\[a_1\otimes \ldots\otimes a_n\otimes b_1\ldots b_n\mapsto (a_1\wedge b_1)\otimes \ldots \otimes (a_n\wedge b_n)\in (\Lambda^2\S_L)^{\otimes n}.\]
Now this section vanishes at $(L, x, D)$ if and only if  $b_i$ is proportional to $a_i$ for some $i=1, \ldots, n$, that is, $b_i\in \O_\L(-1)_{(L,x)}$ for some $i$. If we write $D=\sum_{i=1}^n x_i$ then the fiber of $\O_{\L^{[n]}}(-1)$ over $(L, D)$ is
\[\O_{\L^{[n]}}(-1)_{(L,D)}=\left\{b_1\ldots b_n: b_i\in \O_{\L}(-1)_{(L, x_i)}\right\}\subseteq \Sym^n \S_L\]
But then the condition that $b_i\in \O_\L(-1)_{(L,x)}$ for some $i$ is equivalent to $x=x_i$ for some $i$, that is, $x\in D$ which is precisely the defining condition of $\D$. 

\end{proof}

\subsection{Obstruction bundle and virtual fundamental class}

We can identify the obstruction bundle of $P_{n+1}(X, \beta)$ as follows. By \cite[Proposition 4.6]{PTfoundations} the obstruction bundle has fiber over $(L, D)\in \L^{[n]}$ given by $H^0(\O_D(D)\otimes K_X)^\vee$ (note that $H^1(\N_{L/X})=0$ for any $L$ by \cite[Lemma 1.9]{cubic}). 

In other words,
\[\textup{Obs}=p_\ast(\O_{\D}(\D)\otimes K_X)^\vee=Rp_\ast(\O_{\D}(\D)\otimes K_X)^\vee.\]

We now compute $\textup{Obs}$ in the $K$-theory of $P_{n+1}(X, \beta)=\L^{[n]}$.
 We have $K_X=\O_X(-2H)$ in $X$, so the pullback of $K_X$ to $\L \times_{F(X)}\L^{[n]}$ is $\O(-2\zeta_1)$. We also have 
 \[\O_\D(\D)=\O(\D)-\O=\O(\zeta_n+n\zeta_1+nc_1)-\O\]
in $K(\L\times_{F(X)}\L^{[n]})$, by proposition \ref{universaldivisor}. Thus 
\begin{align*}\textup{Obs}^\vee&=Rp_\ast\left(\O(\zeta_n+(n-2)\zeta_1+nc_1)-\O(-2\zeta_1)\right)\\
&=\O(\zeta_n+nc_1)\otimes Rp_\ast \O((n-2)\zeta_1)-Rp_\ast\O(-2\zeta_1)\end{align*}
in $K(\L^{[n]})$.
Since $\O((n-2)\zeta_1)$ is the pullback of $\O_{\L}(n-2)$ via $q$ we have
\[Rp_\ast \O((n-2)\zeta_1)=\pi_n^\ast R\pi_{1\ast}\O_{\L}(n-2)=\pi_n^\ast \Sym^{n-2}\S^\vee\]
by \cite[Lemma 30.8.4]{stacks}.

Similarly, 
\[Rp_\ast \O(-2\zeta_1)=-\pi_n^\ast \Lambda^2 \S.\]

\begin{proposition}
For $n>0$ the obstruction bundle of $P_{n+1}(X, \beta)=\L^{[n]}$ is
\begin{equation}\left(\O(\zeta_n+nc_1)\otimes \pi_n^\ast\Sym^{n-2}(\S^\vee)\oplus \pi_n^\ast \Lambda^2 \S\right)^\vee\in K(\L^{[n]}).
\end{equation}

In particular,
\begin{equation}[P_{n+1}(X, \beta)]^\textup{vir}=(-1)^n c_1\left(\zeta_n^{n-1}+\frac{(n+2)(n-1)}{2}\zeta_n^{n-2}c_1\right).
\end{equation}
\end{proposition}
\begin{proof}
The identification of the obstruction bundle was done before. Letting $\alpha, \beta$ be the Chern roots of $\S$ it follows that
\begin{equation*}[P_{n+1}(X, \beta)]^\textup{vir}=c_n(\textup{Obs})=(-1)^nc_1\prod_{j=0}^{n-2}\left(\zeta_n+nc_1-j\alpha-(n-2-j)\beta\right).
\end{equation*}
The result is obtained from using $\alpha+\beta=c_1$ and noting that homogeneous polynomials in $c_1, \alpha, \beta$ of degree at least 3 vanish because $H^{>4}(F(X))=0$.
\end{proof}

\subsection{Descendents}

The universal stable pair $\FF$ is given by $\iota_\ast \O(\D)$ where \[\iota: \L\underset{F(X)}{\times} \L^{[n]}\hookrightarrow X\times \L^{[n]}\]
is the canonical inclusion of the universal curve. By Grothendieck-Riemann-Roch and proposition \ref{universaldivisor} it follows that
\begin{equation}\ch(\FF)=\iota_\ast \left(e^{\zeta_n+n\zeta_1+nc_1}\td(-\N_\iota)\right)=e^{\zeta_n+nH+nc_1}\iota_\ast \td(-\N_\iota).\label{descendentsequation}
\end{equation}
Writing $\ch(\gamma)$ for the sum of the descendents $\sum_{k}\ch_k(\gamma)$ and $\ch^0(\gamma)$ for the same sum when $n=0$ (the descendents calculated in \ref{descendentsn0}) we get the following expression for the descendents in terms of descentents with $n=0$:
\[\ch(\gamma)=e^{\zeta_n+nc_1}\ch^0(e^{nH}\gamma).\]

Combining this with \ref{descendentsn0} gives a full computation of the descendents. By the expression of the virtual fundamental class, for $n>0$ it's enough to compute the descendents modulo the ideal
\[\R=H^{>4}(\L^{[n]})+H^{>2}(F(X))[\zeta_n]\subseteq H^\ast(\L^{[n]}).\]
All the following equalities should be understood modulo $\R$:

\begin{align}
&\ch_3(1)=n,\quad
\ch_2(H)=1\\
&\ch_4(1)=c_1\left(\frac{1}{6}+\frac{n}{2}+\frac{n^2}{2}\right)+n\zeta_n,\quad
\ch_3(H)=\frac{1}{2}c_1+\zeta_n,\quad
\ch_2(H^2)=-c_1\\
&\ch_5(1)=c_1 \zeta_n \left(\frac{1}{6}+\frac{n}{2}+\frac{n^2}{2}\right)+n\zeta_n^2,\quad
\ch_4(H)=\frac{1}{2}c_1\zeta_n+\frac{1}{2}\zeta_n^2\\
&\ch_3(H^2)=-c_1\zeta_n,\quad
\ch_2(H^3)=0\\
&\ch_2(\gamma)=\varphi(\gamma),\quad \ch_3(\gamma)=\zeta_n\varphi(\gamma).
\end{align}
Here $\gamma\in H^3(X)$. 

\subsection{Descendent invariants}

We're now in conditions to compute all the descendent invariants. To illustrate the type of expressions arising we'll compute $Z_{\textup{PT}}(\ch_5(1))$. All the remaining equations in theorem \ref{fullPTcubic} are calculated in the same way. We have for $n>0$

\begin{align*}
\langle\ch_5(1)\rangle_{n+1, \beta}^X&=\int_{[P_{n+1}(X, \beta)]^\textup{vir}}\ch_5(1)\\
&=\int_{\L^{[n]}}(-1)^nc_1\left(\zeta_n^{n-1}+\frac{(n+2)(n-1)}{2}c_1\zeta_n^{n-2}\right)\left(c_1 \zeta_n \left(\frac{1}{6}+\frac{n}{2}+\frac{n^2}{2}\right)+n\zeta_n^2\right)\\
&=(-1)^n\frac{45}{2}(3+n^2).
\end{align*}
The last computation is done using that
\[\int_{\L^{[n]}}c_1^2\zeta_n^n=45 \textup{ and }\int_{\L^{[n]}}c_1\zeta_n^{n+1}=-45\frac{n(n+1)}{2}.\]
The first integral is clear by \eqref{integralsfano} and the second is reduced to the first using
\[\zeta_n^{n+1}=-c_1(\Sym^n\S)\zeta_n^n-c_2(\Sym^n\S)\zeta_n^{n-1}\]
and \[c_1(\Sym^n\S)=\frac{n(n+1)}{2}c_1.\]

Thus
\[Z_{\textup{PT}}(\ch_5(1))=\frac{15q}{4}+\sum_{n=1}^\infty (-1)^n\frac{45}{2}(3+n^2)q^{n+1}=\frac{15(q-5q^2+5q^3-q^4)}{4(1+q)^3}.\]

\bibliographystyle{plain}
\bibliography{bibliographycubicPT}


\end{document}